\documentclass[oneside]{amsart} 

\usepackage{calligra}
\DeclareMathAlphabet{\mathcalligra}{T1}{calligra}{m}{n}
\DeclareFontFamily{OT1}{pzc}{}
\DeclareFontShape{OT1}{pzc}{m}{it}{<-> s * [0.900] pzcmi7t}{}
\DeclareMathAlphabet{\mathpzc}{OT1}{pzc}{m}{it}
\usepackage[utf8]{inputenc} 
\usepackage[english]{babel}
\usepackage{amssymb}
\usepackage{amsmath}
\usepackage{amsthm}
\usepackage{amssymb}
\usepackage[all]{xy}
\usepackage[T1]{fontenc}
\usepackage{geometry}
\usepackage{amsfonts}
\usepackage{pdfpages}
\usepackage{mathabx}
\usepackage{mathrsfs}
\usepackage{enumitem}
\usepackage[toc,page]{appendix}
\usepackage{hyperref}
\usepackage{amsmath,calligra,mathrsfs}
\DeclareMathOperator{\Hom}{\mathscr{H}\text{\kern -3pt {\calligra\large om}}\,}

\begin{document}

\title{Mixed Hodge modules without slope}

\author{Matthieu Kochersperger}

\address{UMR 7640 du CNRS, Centre de Mathématiques Laurent Schwartz, École polytechnique,
F–91128 Palaiseau cedex, France}

\email{matthieu.kochersperger@polytechnique.edu}



\maketitle

\date{}

\let\tilde\widetilde
\let\hat\widehat
\let\bar\overline

\theoremstyle{plain}
\newtheorem{theorem}{Theorem}[section]
\newtheorem{lemma}[theorem]{Lemma}
\newtheorem{proposition}[theorem]{Proposition}
\newtheorem{corollary}[theorem]{Corollary}
\theoremstyle{definition}
\newtheorem{de}[theorem]{Definition}
\newtheorem{ccl}[theorem]{Conclusion}
\theoremstyle{remark}
\newtheorem{expl}[theorem]{Example}
\newtheorem{Rem}[theorem]{Remark}
\theoremstyle{plain}
\newtheorem{theo}[theorem]{Theorem}
\newtheorem{lemm}[theorem]{Lemma}
\newtheorem{prop}[theorem]{Proposition}
\newtheorem{coro}[theorem]{Corollary}
\theoremstyle{definition}
\newtheorem{defi}[theorem]{Definition}
\theoremstyle{remark}
\newtheorem{exem}[theorem]{Example}
\newtheorem{rema}[theorem]{Remark}
\theoremstyle{definition}
\newtheorem{ques}[theorem]{Question}
\theoremstyle{definition}
\newtheorem*{defi*}{Definition}
\newtheorem*{remark*}{Remark}
\theoremstyle{plain}
\newtheorem*{prop*}{Proposition}
\newtheorem*{coro*}{Corollary}
\newtheorem{conja}{Conjecture}
\renewcommand\theconja{\Alph{conja}}
\newtheorem{theorema}{Theorem}
\renewcommand\thetheorema{\Alph{theorema}}
\def\restriction#1#2{\mathchoice
              {\setbox1\hbox{${\displaystyle #1}_{\scriptstyle #2}$}
              \restrictionaux{#1}{#2}}
              {\setbox1\hbox{${\textstyle #1}_{\scriptstyle #2}$}
              \restrictionaux{#1}{#2}}
              {\setbox1\hbox{${\scriptstyle #1}_{\scriptscriptstyle #2}$}
              \restrictionaux{#1}{#2}}
              {\setbox1\hbox{${\scriptscriptstyle #1}_{\scriptscriptstyle #2}$}
              \restrictionaux{#1}{#2}}}
\def\restrictionaux#1#2{{#1\,\smash{\vrule height .8\ht1 depth .85\dp1}}_{\,#2}} 

\begin{abstract} 

In this article we are interested in morphisms without slope for mixed Hodge modules. We first show the commutativity of iterated nearby cycles and vanishing cycles applied to a mixed Hodge module in the case of a morphism without slope. Then we define the notion "strictly without slope" for a mixed Hodge module and we show the preservation of this condition under the direct image by a proper morphism. As an application we prove the compatibility of the Hodge filtration and Kashiwara-Malgrange filtrations for some pure Hodge modules with support an hypersurface with quasi-ordinary singularities.

\end{abstract}

\tableofcontents

\reversemarginpar

\part*{Introduction}

In this article we generalise the work of P. Maisonobe in \cite{Maisonobe} on morphisms
\emph{without slopes} for $\mathscr{D}$-modules to the category of mixed Hodge modules of 
M. Saito.

\section*{Morphisms without slope}

The classical example of Lê shows that in general, for a morphism $\boldsymbol{f}:\mathbb{C}^n\to\mathbb{C}^p$ with $p\geq 2$, there is no Milnor fibration:
\[\begin{array}{ccccc}
\boldsymbol{f} & : & (\mathbb{C}^3,0) & \to & (\mathbb{C}^2,0) \\
&& (x,y,z) & \mapsto & (x^2-y^2z,y).
\end{array}
\]
Indeed it can be shown that there is no closed analytic subset $F$ in $\mathbb{C}^2$ such that the number of connected components of the fiber of a representative of the germ of $\boldsymbol{f}$ is constant on $(\mathbb{C}^2-F,0)$.

In order to study morphisms $\boldsymbol{f}:\mathbb{C}^n\to\mathbb{C}^p$ with $p\geq 2$ we must then add conditions. In \cite{HMS} J. P. Henry, M. Merle and C. Sabbah introduce for the morphism $\boldsymbol{f}=(f_1,\dots,f_p)$ the condition of being \emph{without blow-up in codimension $0$}. Under this condition they manage to construct stratifications of the morphism satisfying Whitney and Thom type conditions. Following ideas of Lê D. T. (\cite{Le}), with such a stratification we get Milnor type fibrations above the strata of the target of the morphism. In \cite{SabbahMaS} C. Sabbah shows that, for any analytical morphism $f$, after a proper sequence of blowings-up of the target we can reduce to the case without blow-up in codimension $0$.

 If furthermore the critical locus of $f$ is a subset of the union of the hypersurfaces $f_j^{-1}(0)$ for $j=1,\dots,p$, following \cite{BMM2}, the morphism  is called \emph{without slope}. For a $\mathscr{D}$-module $\mathcal{M}$ or a perverse sheaf $\mathcal{F}$ they also define what it means for pairs $(\boldsymbol{f},\mathcal{M})$ or $(\boldsymbol{f},\mathcal{F})$ to be \emph{without slope}. In this case the former stratification of the target is the usual stratification associated with the coordinate hyperplanes. 

In \cite{Maisonobe} P. Maisonobe proves important results about these morphisms. He shows that if $\mathcal{M}$ corresponds to $\mathcal{F}$ via the Riemann-Hilbert correspondence, i.e. $\mathbf{DR}(\mathcal{M})\simeq \mathcal{F}$, then it is equivalent to have $(\boldsymbol{f},\mathcal{M})$ without slope or $(\boldsymbol{f},\mathcal{F})$ without slope. He defines the functors $\Psi^{alg}_{\boldsymbol{f}}$, $\Phi^{alg}_{\boldsymbol{f}}$ and $\Psi_{\boldsymbol{f}}$ for pairs without slope. 
He shows that as for the Milnor fibration we have in this case
\[\Psi_{\boldsymbol{f}}(\mathcal{F})_0\simeq R\Gamma(B_\epsilon\cap f^{-1}(D_\eta^*)\cap f^{-1}(t),\mathcal{F})\]
for $t\in D_\eta^*$ and $0<\eta\ll\epsilon\ll 1$, where $D_\eta^*$ is a polydisc of radius $\eta$ minus the coordinate hyperplanes.
He then shows that there exist natural isomorphisms
\[\begin{array}{ccc}
\Psi_{\boldsymbol{f}}\mathcal{F} & \simeq & \Psi_{f_1}\cdots\Psi_{f_p}\mathcal{F}, \\
\Psi^{alg}_{\boldsymbol{f}}\mathcal{M} & \simeq & \Psi^{alg}_{f_1}\cdots\Psi^{alg}_{f_p}\mathcal{M}, \\
\Phi^{alg}_{\boldsymbol{f}}\mathcal{M} & \simeq & \Phi^{alg}_{f_1}\cdots\Phi^{alg}_{f_p}\mathcal{M}.
\end{array}
\]

\section*{Aim of the article}

This work has its origins in the work of P. Maisonobe on $\mathscr{D}$-modules without slopes in \cite{Maisonobe} on the one hand, and in results of M. Saito on mixed Hodge modules of normal crossing type in \cite{HM1} and \cite{SaitoMHM} on the other hand. Normal crossing being a special case of a morphism without slope and having in mind the results and technics of P. Maisonobe, our aim is to extend the properties of mixed Hodge modules of normal crossing type to the without slope case. More precisely we give results on the compatibility of the filtrations involved, namely the Kashiwara-Malgragne filtrations and the Hodge filtration. 

In \cite{SaitoMHM}, M. Saito shows that, for a mixed Hodge module of normal crossing type, the Kashiwara-Malgrange filtrations and the Hodge filtration are compatible in the sense of definition \ref{Fcomp}. Using P. Maisonobe results, we get the following proposition 

\begin{prop*}[\cite{kocher1} Proposition 2.12]

If $((f_1,\dots,f_p),\mathcal{M})$ is without slope then the Kashiwara-Malgrange filtrations for the $p$ hypersurfaces $f_j^{-1}(0)$ are compatible in the sense of the definition \ref{Fcomp}.
 
\end{prop*}

This leads to the following conjecture

\begin{conja}\label{conjH}

Let $\mathscr{M}$ be a mixed Hodge module with underlying $\mathscr{D}$-module $\mathcal{M}$ such that the pair $(\boldsymbol{f},\mathcal{M})$ is without slope. Then the Kashiwara-Malgrange filtrations for the $p$ hypersurfaces $f_j^{-1}(0)$ \emph{and the Hodge filtration $F_{\bullet} \mathcal{M}$} are compatible in the sense of the definition \ref{Fcomp}.

\end{conja}

A result of \cite{DMST} gives an evidence for this conjecture. In this article the authors showed the compatibility of the filtrations in the non-characteristic case which is a particular case of being without slope. \\

We will use two approaches to tackle this problem. The first one will consists in using the description of nearby and vanishing cycles functors involving Nilsson classes. The functoriality of this description will enable us to make use of the power of M. Saito's theory of mixed Hodge modules. In this way, we will prove the commutativity of the various nearby and vanishing cycles functors for mixed Hodge modules, using their commutativity for the underlying $\mathscr{D}$-modules.

For the second approach, we will developed the theory of \emph{strictly multispecialisable} $\tilde{\mathscr{D}}$-modules which mimics the theory of $\mathscr{D}$-modules without slope, while taking the Hodge filtration into account. We will highlight the link between this notion and the compatibility of Kashiwara-Malgrange filtrations and the Hodge filtration. In this way we will prove the stability of these properties by a projective direct image. We will also give results on the monodromy filtrations inspired by the normal crossing case. We will then apply these results to quasi-ordinary hypersurface singularities.\\

\subsection*{Mixed Hodge modules without slope}

Following \cite{SaitoMHM}, a mixed Hodge module is made of a triple $\mathscr{M}=(\mathcal{M},F_\bullet\mathcal{M},\mathcal{F})$ and a finite increasing filtration $W_\bullet \mathscr{M}$ satisfying some Hodge conditions where:
\begin{enumerate}
\item $\mathcal{F}$ is a perverse sheaf over $\mathbb{Q}$,
\item $\mathcal{M}$ is a regular holonomic $\mathscr{D}$-module corresponding to $\mathbb{C}\otimes_\mathbb{Q} \mathcal{F}$ via the Riemann-Hilbert correspondence, in other words $\mathbf{DR}(\mathcal{M})\simeq \mathbb{C}\otimes_\mathbb{Q} \mathcal{F}$.
\item $F_\bullet\mathcal{M}$ is a good filtration by $\mathscr{O}$-coherent submodules of $\mathcal{M}$.
\end{enumerate}  
Nearby and vanishing cycles functors are among the main ingredients used in the inductive definition of Hodge modules by M. Saito. We give results on the relations between Kashiwara-Malgrange filtrations and the Hodge filtration $F_\bullet\mathcal{M}$ for a mixed Hodge module $\mathscr{M}$ such that the pair $(\boldsymbol{f},\mathcal{M})$ is without slope. 

In the part \ref{section1} we prove a theorem of commutation of the nearby and vanishing cycles applied to Hodge modules. 

\begin{theorema}[Corollary \ref{theocommuthodge}]\label{commutMHM}

Let $\mathscr{M}$ be a mixed Hodge module and let us suppose that $(\boldsymbol{f},\mathcal{M})$ without slope where $\mathcal{M}$ is the $\mathscr{D}$-module underlying $\mathscr{M}$. Then we have the following isomorphisms
\[\begin{array}{ccc}
\Psi_{f_1}^{HM}\cdots\Psi_{f_p}^{HM}\mathscr{M} & \simeq & 
\Psi_{f_{\sigma(1)}}^{HM}\cdots\Psi_{f_{\sigma(p)}}^{HM}\mathscr{M}, \\
\Phi_{f_1}^{HM}\cdots\Phi_{f_p}^{HM}\mathscr{M} & \simeq & 
\Phi_{f_{\sigma(1)}}^{HM}\cdots\Phi_{f_{\sigma(p)}}^{HM}\mathscr{M}
\end{array}
\]
where $\sigma$ is a permutation of $\{1,\dots,p\}$ and $\Psi_{\boldsymbol{f}}^{HM}$ and $\Phi_{\boldsymbol{f}}^{HM}$ are the nearby and vanishing cycles functors in the category of mixed Hodge modules.
\end{theorema}

\begin{remark*}

If Conjecture \ref{conjH} were true, then this theorem would be an easy consequence of it.

\end{remark*}

 In Part \ref{section2}, mimicking the definition of being without slope for a pair $(\boldsymbol{f},\mathcal{M})$, we give a definition of being \emph{strictly multispecialisable} for the pair $(\boldsymbol{f},\mathscr{M})$ and we prove the following theorem

\begin{theorema}[Propositions \ref{compatstrict} et \ref{strictmulcompat}]

Let $\mathscr{M}$ be a mixed Hodge module and $\boldsymbol{f}$ a morphism such that the pair $(\boldsymbol{f},\mathcal{M})$ is without slope, then the following statements are equivalent:
\begin{enumerate}
\item the pair $(\boldsymbol{f},\mathscr{M})$ is strictly multispecialisable
\item the filtrations $(F_\bullet\mathcal{M},V_1^\bullet\mathcal{M},\dots,V_p^\bullet\mathcal{M})$ are compatible.
\end{enumerate}
Here $V_j^\bullet\mathcal{M}$ is the Kashiwara-Malgrange filtration for the hypersurface $f_j^{-1}(0)$.
\end{theorema}

We then prove two theorems of push-forward for mixed Hodge modules strictly multispecialisable.

\begin{theorema}[Theorem \ref{commutVgrH}]

Let $\mathscr{M}$ be a mixed Hodge module. If $\mathscr{M}$ is strictly multispecialisable then the push-forward of $\mathscr{M}$ by a proper morphism is also strictly multispecialisable. Moreover, in this case, the Kashiwara-Malgrange \emph{multi}filtration, the nearby cycles $\Psi_{\boldsymbol{f}}^{HM}$ and the vanishing cycles $\Phi_{\boldsymbol{f}}^{HM}$ commute with the push-forward. 

\end{theorema}

\begin{defi*}

Let $A$ be an object in an abelian category equipped with $p$ nilpotent endomorphisms  $N_1,\dots,N_p$. We say that $(A,N_1,\dots,N_p)$ satisfies property $(MF)$ if the relative monodromy filtrations satisfy 
\[W(N_1,W(\dots,W(N_p)))A=W(N_1+\dots+N_p)A.\] 
(See section \ref{relatmonofilt} for definitions.)

\end{defi*}

\begin{theorema}[Theorem \ref{imagedirecteW}]

Let $\mathscr{M}$ be a mixed Hodge module. Let $f:X\times \mathbb{C}^p_{\boldsymbol{t}}\to Y\times \mathbb{C}^p_{\boldsymbol{t}}$ be a projective morphism and suppose that $(\mathscr{M},N_1,\dots,N_p)$ satisfy property (MF), then $(\Psi_{\boldsymbol{t}}\mathscr{M},N_1,\dots,N_p)$ also satisfy property (MF).

\end{theorema}

As an application we prove conjecture \ref{conjH} for quasi-ordinary hypersurface singularities.

\section*{Notation}

\subsubsection*{Sheaves}
Let $X$ be a complex variety, we will consider the following sheaves:

\begin{itemize}
\item $\mathcal{O}_X$ the sheaf of holomorphic functions on $X$,
\item $\Omega^p_X$ the sheaf of holomorphic $p$-forms on $X$,
\item $\mathscr{D}_X$ the sheaf of differential operators on $X$,
\item $F_\bullet\mathscr{D}_X$ the order filtration of $\mathscr{D}_X$, 
\item $\tilde{\mathscr{O}}_X:=R_F\mathscr{O}_X$ the sheaf associated to  $(\mathscr{O}_X,F_\bullet\mathscr{O}_X)$ by Rees's construction \ref{Rees}.
\item $\tilde{\mathscr{D}}_X:=R_F\mathscr{D}_X$ the sheaf associated to  $(\mathscr{D}_X,F_\bullet\mathscr{D}_X)$ by Rees's construction \ref{Rees}.

\end{itemize}

\subsubsection*{Categories}
Let $X$ be a complex variety, we will consider the following categories:
\begin{itemize}
\item the category of perverse sheaves $\textup{Perv}(X)$,
\item the category of regular holonomic $\mathscr{D}_X$-modules $\textup{Mod}_{h,r}(\mathscr{D}_X)$,
\item the category of graded holonomic $\tilde{\mathscr{D}}_X$-modules $\textup{Mod}_{r}(\tilde{\mathscr{D}}_X)$,
\item the bounded derived category of sheaves with constructible cohomology $D_c^b(X)$,
\item the bounded derived category of $\mathscr{D}_X$-modules with regular holonomic cohomology $D_{h,r}^b(\mathscr{D}_X)$,
\item the category of polarisable pure Hodge modules of weight $k$ $\textup{MH}^p(X,\mathbb{Q},k)$,
\item the category of mixed Hodge modules $\textup{MHM}(X,\mathbb{Q})$, we will always suppose that the graded pieces for the weight filtration are polarisable.

\end{itemize}

\subsubsection*{Objects}

We will denote the objects of the above categories in the following way:
\begin{itemize}
\item $\mathcal{F},\mathcal{G}\in D_c^b(X)$,
\item $\mathcal{M},\mathcal{N}\in \textup{Mod}_{h,r}(\mathscr{D}_X)$,
\item $\mathfrak{M},\mathfrak{N}\in \textup{Mod}_{r}(\tilde{\mathscr{D}}_X)$,
\item $\mathscr{M},\mathscr{N}\in \textup{MHM}(X,\mathbb{Q})$.

\end{itemize}

\subsubsection*{Multi-indexes}

Let $X$ be a complex variety of dimension $d$. Let $\{H_i\}_{1\leq i\leq d}$ be $p$ smooth hypersurfaces given by the equations $t_i=0$. We will use the following notations:

\begin{itemize}

\item $\partial_{i}  :=  \partial_{t_i}\in\mathscr{D}_X$,
\item $\eth_{i}:=\eth_{t_i}=z\partial_i\in\tilde{\mathscr{D}}_X$,
\item $E_i  :=  t_i\partial_{i}$,
\item $\boldsymbol{x}:=(x_1,\dots,x_{d-p})$,
\item $\boldsymbol{1}_i:=(0,\dots,0,1,0,\dots,0)$ where $1$ is in the $i$-th place,
\item $\boldsymbol{\alpha}:=(\alpha_1,\dots,\alpha_p)$,
\item $\boldsymbol{\alpha}_{I}:=(\alpha_i)_{i\in I}$ for $I\subset \{1,\dots,p\}$,
\item $\boldsymbol{t}:=t_1\dots t_p$,
\item $\boldsymbol{t}^{{\boldsymbol{s}}}:=t_1^{s_1}\dots t_p^{s_p}$,
\item $\mathscr{D}_X[{\boldsymbol{s}}]:=\mathscr{D}_X[s_1,\dots,s_p]$.
\item Let $V_{\boldsymbol{\alpha}}\mathcal{M}$ be a multifiltration, we denote
\[\textup{gr}_{\boldsymbol{\alpha}}^V:=\frac{V_{\boldsymbol{\alpha}}}
{\sum_{1\leq i\leq p}V_{\alpha_1 \dots <\alpha_i \dots \alpha_p}}\]
and for for $I\subset \{1,\dots,p\}$,
\[V_{<\boldsymbol{\alpha_{I}},\boldsymbol{\alpha}_{I^c}}:=\sum_{i\in I}V_{\alpha_1 \dots <\alpha_i \dots \alpha_p}.
\]
\end{itemize}

\part{Commutativity of nearby and vanishing cycles}\label{section1}

In this part, we prove the following commutativity result for mixed Hodge modules:

\textit{Let $\mathscr{M}\in \textup{MHM}(X\times{\Delta}^p_{\boldsymbol{t}})$ be a mixed Hodge module on $X\times{\Delta}^p_{\boldsymbol{t}}$ such that the pair $(\mathbf{H},\mathcal{M})$ is without slope where $\mathcal{M}$ is the right $\mathscr{D}_{X\times\Delta^p_{\boldsymbol{t}}}$-module underlying $\mathscr{M}$. 
Then, for all permutation $\sigma$ of $\{1,\dots ,p\}$, we have an isomorphism}
\[\Psi_{t_1}\cdots \Psi_{t_p}\mathscr{M} \simeq \Psi_{t_{\sigma(1)}}\cdots \Psi_{t_{\sigma(p)}}\mathscr{M}.\]

We are going to use commutativity results of \cite[corollary 2.13]{kocher1} for the  $\mathscr{D}$-module underlying $\mathscr{M}$. However, this proposition enables us to compare iterated vanishing cycles through vanishing cycles for several functions. A priori, we cannot equip the latter with a mixed Hodge module structure. That's why we will rather use the definition of algebraic nearby cycles that uses Nilsson class functions module. With this definition we will be able to use the powerful functorial formalism of the category of mixed Hodge modules.

Thus, we will start by endowing the Nilsson class functions module with a mixed Hodge module structure. Then, in order to make the link between the graded pieces of the $V$-multifiltration and the Nilsson classes, we will define the morphism $\mathbf{Nils}$  for mixed Hodge modules. This will enable us to construct a comparison morphism between iterated vanishing cycles \emph{in the category of mixed Hodge modules}. We will then use results on $\mathscr{D}$-modules thanks to proposition 
\ref{isoDH}.

\begin{rema}\label{defiprodtens}

In this part, we will indicate the category we consider on the bottom left of each functor. We will use the letters $D$, $\tilde{D}$, $P$ and $H$ to denote respectively the categories $\textup{Mod}(\mathscr{D}_X)$, $\textup{Mod}(\tilde{\mathscr{D}}_X)$, $\textup{Perv}(\mathbb{Q}_X)$ and $\textup{MHM}(X)$. For instance, the inverse image functor by a smooth morphism $f$ in the category of perverse sheaves over $\mathbb{Q}$ will be denoted $_Pf^*$ and the tensor product functor in the category of mixed Hodge modules will be denoted  $_H\otimes$.

In the category $\textup{MHM}(X)$, the tensor product of $\mathscr{M}$ and $\mathscr{N}$ is defined by the following formula:
\[\mathscr{M}{_H\otimes}~\mathscr{N}:={_H\delta}^*\left(\mathscr{M}{_H\boxtimes}~\mathscr{N}\right)
\]
where $\delta:X\to X\times X$ is the inclusion of the diagonal.

\end{rema}

\section{The mixed Hodge module $\mathscr{N}_{q,k}$}

In order to make the link between nearby cycles and the Nilsson class, the first step in the proof of the commutativity theorem is to define the following morphism in the category of mixed Hodge modules
\[
\mathbf{Nils}:\Psi_{t}\mathscr{M} \to  {_Hi_*}\mathscr{H}^{-1}{_Hi^*}\left(\mathscr{M}_{q,k}\right)\]
where 
\[\mathscr{M}_{q,k}:=\mathscr{M}~{_H\otimes}
~{_H\pi^{!}}\left(\mathscr{N}_{q,k}\right).\]
with  
\[\begin{array}{cccc}
     \pi:& X\times {\Delta} & \to & {\Delta}\\
       & (x,t) & \mapsto & t
     \end{array}.\] 
Accordingly, in this section we define the mixed Hodge module $\mathscr{N}_{q,k}$.

Let $\Delta\subset \mathbb{C}$ be a small disc with center $0$ and coordinate $s$. We denote $\Delta^*:=\Delta-\{0\}$. We start by defining a polarized mixed Hodge structure  for all $k\in \mathbb{N}$.

\begin{defi}

Let $k\in\mathbb{N}$ and let $V_k$ be the $\mathbb{Q}$-vector space with basis $\{v_\ell\}_{0\leq \ell\leq k}$, we define:
\begin{itemize} 
\item the decreasing filtration $F^pV_{k,\mathbb{C}}:=\text{Vect}(\{v_\ell\}_{p\leq \ell\leq k})$,
\item the nilpotent endomorphism $N$ satisfying $N(v_\ell)=v_{\ell-1}$ for all $\ell$ such that $0\leq \ell\leq k$.
\item the bilinear form $Q$ satisfying $Q(v_\ell,v_{k-\ell'})=\delta_{\ell\ell'}$ for all $\ell$ such that $0\leq \ell\leq k$.
\end{itemize}
The weight filtration of this polarized mixed Hodge structure is given by the formula
\[W_iV_k=\text{Vect}(\{v_\ell\}_{0\leq \ell\leq i/2}).\]

\end{defi}

\begin{rema}

This mixed Hodge structure satisfies the following equality of Hodge structures of weight $2k$
\[P_k(V_k)\simeq\textup{gr}_{2k}^W V_k=\mathbb{Q}(-k).
\]

\end{rema}

Following \cite[Corollary (3.13)]{CKS2}, the nilpotent orbit associated to this polarized mixed Hodge structure is a polarized variation of Hodge structure of weight $k$ on $\Delta^*$, we denote it ${\mathcal{H}}_k$. 
The corresponding $\mathbb{Q}$-local system is given by the monodromy endomorphism $\exp(N)$. The Hodge filtration is given by the following formula
\[F^p{\mathcal{H}}_k:=\exp\left(\frac{\log s}{2i\pi}N\right)F^pV_k.
\]
We denote $\{e_\ell\}_{0\leq \ell\leq k}$ 
the basis of holomorphic sections defined by
\[e_\ell:=(2i\pi)^\ell\sum_{0\leq j\leq \ell}\frac{1}{j!}\left(\frac{\log s}{2i\pi}\right)^jv_{\ell-j},
\]
we have
\[s\partial_s\cdot e_{\ell}=e_{\ell-1},
\]
and
\[F^p{\mathcal{H}}_k=\bigoplus_{p\leq \ell\leq k}\mathcal{O}_{\Delta^*}e_\ell.
\]
We switch to the right $\mathscr{D}_{\Delta^*}$-modules convention and we denote $\mathscr{H}_k$ the polarized Hodge module corresponding to the polarized variation of Hodge structure ${\mathcal{H}}_k$. 

\begin{defi}

We define $\mathscr{N}_{k}\in \textup{MHM}(\Delta)$ as the localisation of $\mathscr{H}_k$
\[\mathscr{N}_{k}:=\mathscr{H}_k[*0].
\]

\end{defi}

The right $\tilde{\mathscr{D}}_{\Delta}$-module underlying $\mathscr{N}_k$ is spanned by the sections $\left\{\frac{\tilde{e}_\ell}{s}\right\}_{0\leq \ell\leq k}$, where $\tilde{e}_{\ell}:=\frac{{e}_{\ell}}{z^\ell}$. We have 
\[\frac{\tilde{e}_\ell}{s}\cdot s\eth_s=-\frac{\tilde{e}_{\ell-1}}{s}.
\]
Let us consider the application
\[\begin{array}{cccc}
     r_q:&  \Delta & \to & {\Delta}\\
       & s & \mapsto & t=s^q
     \end{array}.\]

\begin{defi}\label{MHMNilss}

We define a polarized mixed Hodge module:
\[\mathscr{N}_{q,k}:={_Hr_{q*}}\mathscr{N}_k.\]
we denote $V^{\mathbb{Q}}_{q,k}$ the perverse sheaf underlying $\mathscr{N}_{q,k}$ and
$\{v_{p,\ell}\}_{\substack{0\leq\ell\leq k \\ 1\leq p\leq q}}$ basis of $_P\Psi_tV^{\mathbb{Q}}_{q,k}$ coming from the basis $\{v_\ell\}_{0\leq \ell\leq k}$. The right $\tilde{\mathscr{D}}_{\Delta}$-module underlying $\mathscr{N}_{q,k}$ can be written
\[\bigoplus_{1\leq p\leq q} \mathfrak{N}_{-\frac{p}{q},k},\]
where $\mathfrak{N}_{-\frac{p}{q},k}$ is the right $\tilde{\mathscr{D}}_{\Delta}$-module spanned by the sections $\left\{\tilde{e}_{-\frac{p}{q},\ell}\right\}_{0\leq \ell\leq k}$ 
where $\tilde{e}_{-\frac{p}{q},\ell}:=s^{-p}q^\ell\tilde{e}_\ell$. We have

\begin{equation}\label{eqNilssMHM}
\tilde{e}_{-\frac{p}{q},\ell}\cdot t\eth_t = -\left(1-\frac{p}{q}\right)z\cdot\tilde{e}_{-\frac{p}{q},\ell}-\tilde{e}_{-\frac{p}{q},\ell-1}.
\end{equation}

\end{defi}

Thus, for all $\alpha\in\mathbb{Q}\cap[-1,0[$ and for all $k\in\mathbb{N}$,
$\mathfrak{N}_{\alpha,k}$ is a direct factor of a polarizable mixed Hodge module and, by definition, for all $\alpha\in\mathbb{Q}\cap[-1,0[ $ and for all $k\in\mathbb{N}$, $\mathfrak{N}_{\alpha,k}$ is equal to its localization.

\begin{defi}

Let us consider the projection
 \[\begin{array}{cccc}
     \pi:& X\times {\Delta} & \to & {\Delta}\\
       & (x,t) & \mapsto & t
     \end{array}.\] 
and let $\mathscr{M}$ be a mixed Hodge module, we define 
\[\mathscr{M}_{q,k}:=\mathscr{M}~{_H\otimes}
~{_H\pi^{!}}\left(\mathscr{N}_{q,k}\right).\]
We have a decomposition in the category $\textup{Mod}_{r}(\tilde{\mathscr{D}}_{X\times \Delta})$
\[\mathfrak{M}_{q,k}=\bigoplus_{1\leq p\leq q} \mathfrak{M}_{-\frac{p}{q},k}\]
where
\[\mathfrak{M}_{\alpha,k}:=\mathfrak{M}~{_{\tilde{D}}\otimes}~
\left({_{\tilde{D}}\pi^*}\mathfrak{N}_{\alpha,k}\right).
\]
with $\mathfrak{M}$ the $\tilde{\mathscr{D}}_{X\times \Delta}$-module underlying $\mathscr{M}$.

\end{defi}

\section{The $V$-filtration of $\mathfrak{M}_{\alpha,k}$}

In this section, we compute the Kashiwara-Malgrange filtration of  $\mathfrak{M}_{\alpha,k}$ in terms of the Kashiwara-Malgrange filtration of $\mathfrak{M}[*{X_0}]$, where $X_0:=X\times \{0\}\subset X\times \Delta$. We have the equality
\[\mathfrak{M}_{\alpha,{k}}=\bigoplus_{0\leq\ell\leq{k}}\mathfrak{M}\left[*X_0\right]\tilde{e}_{\alpha,\ell}
\]
compatible with the Kashiwara-Malgrange filtration in the following sense.

\begin{prop}\label{VfilttildeMalpha}
 
Let $\alpha\in \mathbb{Q}\cap[-1,0[$ and $k\in \mathbb{N}$, the $\tilde{\mathscr{D}}_{X\times{\Delta}}$-module $\mathfrak{M}_{\alpha,k}$ is strictly $\mathbb{R}$-specialisable along $X_0$ and, for all $\beta\in{\mathbb{R}}$, we have:
\[V_{\beta}(\mathfrak{M}_{\alpha,{k}})=\bigoplus_{0\leq\ell\leq{k}}V_{\beta+\alpha+1}
 \left(\mathfrak{M}\left[*X_0\right]\right)\tilde{e}_{\alpha,\ell}.
\]
with localisation defined in \ref{localisationF}.

\end{prop}

\begin{proof}

According to the above, $\mathfrak{M}_{\alpha,k}$ is a direct factor of a $\tilde{\mathscr{D}}_{X\times{\Delta}}$-module underlyinga polarizable mixed Hodge module. Thus it is strictly $\mathbb{R}$-specialisable along $X_0$.

We denote
\[U_{\beta}(\mathfrak{M}_{\alpha,{k}}):=\bigoplus_{0\leq\ell\leq{k}}V_{\beta+\alpha+1}
 \left(\mathfrak{M}[*X_0]\right)\tilde{e}_{\alpha,\ell}.
\]
 Let $m\in \mathfrak{M}[*X_0]$, $\ell\in \mathbb{N}$ and $\gamma\in {\mathbb{R}}$, we have
\begin{equation}\label{calcul2}
(m\otimes \tilde{e}_{\alpha,\ell})(t\eth_t+\gamma z)=(m(t\eth_t+\gamma z-(\alpha+1)z))\otimes \tilde{e}_{\alpha,\ell} - m\otimes \tilde{e}_{\alpha,\ell-1},
\end{equation}
Thus, for all $\beta\in{\mathbb{R}}$, $U_\beta(\mathfrak{M}_{\alpha,{k}})$ is a $V_0\tilde{\mathscr{D}}_{X\times{\Delta}}$-module.

Let us show that $U_{\beta}(\mathfrak{M}_{\alpha,{k}})\subset V_{\beta}(\mathfrak{M}_{\alpha,{k}})$. Let $m\in V_{\beta+\alpha+1}\left(\mathfrak{M}[*X_0]\right)$, there exist polynomials
\begin{itemize}
\item $b(s)=\prod_{a\in A} (s-az)$ such that for all $a\in A$, $a\leq \beta+\alpha+1$, 
\item $P(s)\in\tilde{\mathscr{D}}_{X\times{\Delta}/{\Delta}}[s]$ with coefficients independent from $\eth_t$ 
\end{itemize}
such that
\[m\cdot(b(t\eth_t)-tP(t\eth_t))=0.\]
We deduce from \eqref{calcul2} that
\[(m\otimes \tilde{e}_{\alpha,\ell})\cdot(b\left(t\eth_t+(\alpha+1)z\right)-tP(t\eth_t+(\alpha+1)z))
\in \bigoplus_{0\leq\ell'<{\ell}}V_{\beta+\alpha+1}
 \left(\mathfrak{M}[*X_0]\right)\tilde{e}_{\alpha,\ell'},
\]
yet the roots of $b\left(s+(\alpha+1)z\right)$ are less or equal than $\beta$, 
thus we show by induction that $(m\otimes \tilde{e}_{\alpha,\ell})\in V_{\beta}(\mathfrak{M}_{\alpha,{k}})$ and therefore that $U_{\beta}(\mathfrak{M}_{\alpha,{k}})\subset V_{\beta}(\mathfrak{M}_{\alpha,{k}})$. 

To prove the converse Let us consider the exact sequence of $V_0(\tilde{\mathscr{D}}_X)$-modules:
\[
0 \to U_\beta(\mathfrak{M}_{{\alpha},{{k}}-1}) \to U_\beta(\mathfrak{M}_{{\alpha},{{k}}}) \to  U_\beta(\mathfrak{M}_{{\alpha},0}) \to 0,
\]
from the above, it is enough to show that $U_\beta(\mathfrak{M}_{{\alpha},0})=V_\beta(\mathfrak{M}_{{\alpha},0})$ to conclude that $U_\beta(\mathfrak{M}_{{\alpha},k})=V_\beta(\mathfrak{M}_{{\alpha},k})$ for all $k\in\mathbb{N}$. Let $(m\otimes \tilde{e}_{\alpha,0})\in V_{\beta}(\mathfrak{M}_{\alpha,0})$, there exists polynomials
\begin{itemize}
\item $b(s)=\prod_{a\in A} (s-az)$ such that for all $a\in A$, $a\leq \beta$, 
\item $P(s)\in\tilde{\mathscr{D}}_{X\times{\Delta}/{\Delta}}[s]$ with coefficients independent from $\eth_t$ 
\end{itemize}
such that
\[(m\otimes \tilde{e}_{\alpha,0})\cdot(b(t\eth_t)-tP(t\eth_t))=0.\]
We deduce from \eqref{calcul2} that 
\[\left[m\cdot(b(t\eth_t-(\alpha+1)z)-tP(t\eth_t-(\alpha+1)z))\right]\otimes \tilde{e}_{\alpha,0}=0,\]
and thus
\[m\cdot(b(t\eth_t-(\alpha+1)z)-tP(t\eth_t-(\alpha+1)z))=0.\]
Finally, $m\in V_{\beta+\alpha+1}
 \left(\mathfrak{M}[*X_0]\right)$, which concludes the proof.

\end{proof}

\section{The morphism Nils}

Here we define the morphism $\bold{Nils}$ for $\tilde{\mathscr{D}}_{X\times{\Delta}}$-modules and for perverse sheaves and then combine both definitions for mixed Hodge modules.

\begin{defi}\label{DNils}

Let $\mathfrak{M}$ be a $\tilde{\mathscr{D}}_{X\times{\Delta}}$-module underlying a polarizable mixed Hodge module. Following proposition \ref{VfilttildeMalpha}, we define the following morphism

\[\begin{array}{cccc}
{_{\tilde{{D}}}\mathbf{N}}: & \text{gr}^V_{{\alpha}}\left(\mathfrak{M}\right) & \longrightarrow & \text{gr}^V_{-{1}}(\mathfrak{M}_{{\alpha},{{k}}}) \\
& m & \longmapsto & \displaystyle\sum_{0\leq {\ell}\leq{{k}}}\left[m\cdot(t\eth_t-\alpha z)^{\ell}\right]\otimes \tilde{e}_{{\alpha},{\ell}}.
\end{array}\]

\end{defi}

\begin{defi}\label{PervNils}

We consider here the rational perverse sheaf $V^{\mathbb{Q}}_{q,k}$ of definition \ref{MHMNilss}. Let $K\in{Perv}(\mathbb{Q}_{X\times{\Delta}})$ be a perverse sheaf on $X\times{\Delta}$. We define ${_{P}}\mathbf{N}:  {_{P}\Psi_t(K)}  \longrightarrow  {_{P}\Psi_{t,1}}(K~{_{P}\otimes}~ {_{P}\pi^!} V^{\mathbb{Q}}_{q,k}[1])$ as the morphism of perverse sheaves
\[\displaystyle\sum_{\substack{  0 \leq \ell \leq  k \\
 1 \leq  p \leq  q
}}
\left({\log T_u} \right)^\ell\otimes v_{p,\ell}
\]
composed with the projection from $_{P}\Psi_t$ to $_{P}\Psi_{t,1}$. Here, $T_u$
is the unipotent part of the monodromy.
\end{defi}

\begin{defi}

Let $\mathscr{M}\in \textup{MHM}(X\times{\Delta})$ be a mixed Hodge module on $X\times{\Delta}$. There exists a finite set of rational numbers in $[-1,0[$, $A$ such that the $V$-filtration of the $\tilde{\mathscr{D}}_{X\times{\Delta}}$-module underlying 
$\mathscr{M}$ is indexed by $A+\mathbb{Z}$. Let $q\in\mathbb{N}$ be such that $A\subset \{-\frac{p}{q}|1\leq p\leq q\}$, we consider
\[\mathscr{M}_{q,k}:=\mathscr{M}~{_H\otimes}
~{_H\pi^{!}}\left(\mathscr{N}_{q,k}\right).\]
If $\mathcal{M}$ is a $\mathscr{D}_{X\times \Delta}$-module, the monodromy endomorphism of ${_{P}\Psi_t\mathbf{DR}_{X\times\Delta}\mathcal{M}}$ correspond to the action of
 $\exp(2i\pi t\partial_t)$ on ${_{D}\Psi_t\mathcal{M}}$. Combining definitions \ref{DNils} and \ref{PervNils}, we get a morphism of mixed Hodge modules
\[_H\mathbf{N}:  \Psi_t(\mathscr{M})  \longrightarrow  \Psi_{t,1}(\mathscr{M}_{q,k}).\]

\end{defi}

\begin{lemm}\label{canul}

Let $k_{\alpha}$ be the exponent of the nilpotent endomorphism of right multiplication by
$(t\eth_t -\alpha z)$ on $\Psi_{t,\lambda}\mathscr{M}$ where $\lambda:=e^{2i\pi\alpha}$.
Then, for all $k\geq\max_{\alpha\in[-1,0[}k_\alpha$, the morphism
\[\textup{can}\circ {_{H}\textup{\textbf{{N}}}} : \Psi_{t}\mathscr{M} \to \Phi_{t,1}\left(\mathscr{M}_{q,k}\right)\]
is equal to zero.

\end{lemm}

\begin{proof}

We check this on the definition of $_{\tilde{D}}\mathbf{N}$. The forgetful functor from the category of mixed Hodge modules to the category of $\mathscr{\tilde{D}}$-modules is faithful, therefore if $\textup{can}\circ {_{\tilde{D}}\textup{\textbf{{N}}}}$ is equal to zero, then $\textup{can}\circ {_{H}\textup{\textbf{{N}}}}$ is also equal to zero.

\end{proof}

\begin{defi}\label{Nils}

Let us consider the exact sequence (2.24.3) fo \cite[2.24]{SaitoMHM}

\begin{equation}\label{SEsaito}0 \to {_Hi_*}\mathscr{H}^{-1}{_Hi^*}\left(\mathscr{M}_{q,k}\right) \to \Psi_{t,1}\left(\mathscr{M}_{q,k}\right) \xrightarrow{\textup{can}} \Phi_{t,1}\left(\mathscr{M}_{q,k}\right) \to
{_Hi_*}\mathscr{H}^{0}{_Hi^*}\left(\mathscr{M}_{q,k}\right) \to 0.
\end{equation}
Conbined with the previous lemma \ref{canul}, this gives a natural morphism
\[\mathbf{Nils}:\Psi_{t}\mathscr{M} \to  {_Hi_*}\mathscr{H}^{-1}{_Hi^*}\left(\mathscr{M}_{q,k}\right).\]

\end{defi}

\begin{rema}\label{defiminv}

Here, $i$ is the inclusion of $X_0\times\{0\}$ in $X\times\Delta$. We denote by $j:X\times \Delta^*\hookrightarrow X\times \Delta$ the inclusion. 
According to definition \ref{defimageinverseH}, ${_Hi_*}\mathscr{H}^{-1}{_Hi^*}\mathscr{M}$ is $\mathscr{H}^{-1}$ of the complex
\[0 \to {_Hj_!}{_Hj^{*}}\mathscr{M} \to \underset{\bullet}{\mathscr{M}} \to 0.\]

\end{rema}

\section{The natural morphism for two hypersurfaces}

In this section we construct a comparison morphism between $\Psi_{t_1}\Psi_{t_2}\mathscr{M}$ and $\Psi_{t_1}\Psi_{t_2}\mathscr{M}$. The intermediate objects occurring in this construction will be defined using usual functors. Thanks to the functorial formalism of M. Saito, the comparison morphism will thus be defined \emph{in the category of mixed Hodge modules}.

\begin{defi}\label{MHMplat}

Let $\mathscr{N}\in \textup{MHM}(X)$, we say that $\mathscr{N}$ is \emph{flat} if its underlying $\mathscr{D}_X$-module is flat as an $\mathscr{O}_X$-module.

\end{defi}

\begin{lemm}

If $\mathscr{N}\in \textup{MHM}(X)$ is flat, then the functor $\cdot{_H\otimes}~\mathscr{N}$ is exact.

\end{lemm}

\begin{proof}

The tensor product functor in the category $\textup{MHM}(X)$ correspond to the tensor product of the underlying $\mathscr{D}_X$-modules through the natural forgetful functor. Let $0\to\mathscr{M}_1\to\mathscr{M}_2\to\mathscr{M}_3\to 0$ be an exact sequence in $\textup{MHM}(X)$ considered as a complex $\mathscr{M}_\bullet$. Since the  
$\mathscr{O}_X$-module underlying $\mathscr{N}$ is flat, for all $i\in \mathbb{Z}$, the $\mathscr{D}_X$-module underlying $\mathscr{H}^i(\mathscr{M}_{\bullet}{_H\otimes}~\mathscr{N})$ is zero. Thus, for all $i\in\mathbb{Z}$, $\mathscr{H}^i(\mathscr{M}_{\bullet}{_H\otimes}~\mathscr{N})$ is equal to zero and  $\mathscr{M}_{\bullet}{_H\otimes}~\mathscr{N}$ is a short exact sequence.

\end{proof}

\begin{lemm}\label{morphismenat}

Let $i:Y\to X$ be the inclusion of a smooth hypersurface, $\mathscr{M}$ and $\mathscr{N}$ objetcs in $\textup{MHM}(X)$ such that $\mathscr{N}$ is flat. There exists a natural morphism
\begin{equation}\label{commuttensor}
{_Hi}_*\mathscr{H}^{-1}{_Hi}^*(\mathscr{M}{_H\otimes}\mathscr{N})\to ({_Hi}_*\mathscr{H}^{-1}{_Hi}^*\mathscr{M}){_H\otimes}\mathscr{N}.\end{equation}

\end{lemm}

\begin{proof}

In this proof, we will omit the index $H$ of functors to lighten notation. We will use the definition of the functor ${i}_*\mathscr{H}^{-1}{i}^*$ in remark \ref{defiminv}. Let us denote by $j:X\setminus Y\to X$ the inclusion. We consider the following commutative square:
\[\xymatrix{
 (j_!j^{*}\mathscr{M})\otimes\mathscr{N}  \ar[r]  & \mathscr{M}\otimes\mathscr{N}  \\ 
 j_!j^{*}\left((j_!j^{*}\mathscr{M})\otimes\mathscr{N}\right) \ar[u]^{\psi} \ar[r]^-{\phi}  & j_!j^{*}(\mathscr{M}\otimes\mathscr{N}).  \ar[u]
}\]
The functor $j_!j^{*}$ depends only on the restriction to $X\setminus Y$, thus the morphism $\phi$ is an isomorphism. We deduce that the morphism of complexes induced by $f$ is a quasi-isomorphism. Therefore we have a morphism of complexes
\[\xymatrix{
0 \ar[r] & (j_!j^{*}\mathscr{M})\otimes\mathscr{N}  \ar[r]  & \mathscr{M}\otimes\mathscr{N} \ar[r] & 0 \\ 
0 \ar[r] & j_!j^{*}(\mathscr{M}\otimes\mathscr{N}) \ar[u]^{\psi\circ\phi^{-1}} \ar[r]  & \underset{\bullet}{\mathscr{M}\otimes\mathscr{N}}  \ar@{=}[u] \ar[r] & 0.
}\]
This induces a morphism between the $\mathscr{H}^{-1}$ of these two complexes, applying the definition of the functor ${i}_*\mathscr{H}^{-1}{i}^*$ in remark \ref{defiminv}, we get the morphism
\[{i}_*\mathscr{H}^{-1}{i}^*(\mathscr{M}\otimes\mathscr{N})\longrightarrow
({i}_*\mathscr{H}^{-1}{i}^*\mathscr{M})\otimes\mathscr{N},
\]
we use here the exactness of the functor $\cdot\otimes\mathscr{N}$.

\end{proof}

%








Let us consider the projection
 \[\begin{array}{cccc}
     \pi=(\pi_1,\pi_2):& X\times {\Delta}^2 & \to & {\Delta}^2\\
       & (x,t_1,t_2) & \mapsto & (t_1,t_2).
     \end{array}\] 
We denote $H_i:=\{t_i=0\}$ for $i\in\{1,2\}$, we have the natural inclusions
\begin{itemize}
\item $i_1:H_1\hookrightarrow X\times {\Delta}^2$, 
\item $i_2:H_2 \hookrightarrow X\times {\Delta}^2 $,
\item $i_{(1,2)}:H_1\cap H_2\hookrightarrow X\times {\Delta}^2$.
\end{itemize}
We define the following polarizable mixed Hodge module on $\Delta^2$:
\[\mathscr{N}_{\boldsymbol{q},\boldsymbol{k}}:=
\mathscr{N}_{q_1,k_1}~{_H\boxtimes}~\mathscr{N}_{q_2,k_2}\]
and
\[\mathscr{M}_{\boldsymbol{q},\boldsymbol{k}}:=
\mathscr{M}~{_H\otimes}~{_H\pi}^!\mathscr{N}_{\boldsymbol{q},\boldsymbol{k}}.
\]

\begin{prop}\label{constructiondemidiag}

Let $\mathscr{M}\in \textup{MHM}(X\times\Delta^2)$, there exist natural morphisms
\begin{equation}\label{demidiag}
\xymatrix{{_Hi}_{1*}\mathscr{H}^{-1}{_Hi}_{1}^*({_Hi}_{2*}\mathscr{H}^{-1}{_Hi}_{2}^*
(\mathscr{M}_{q_1,k_1})~{_H\otimes}~{_H\pi_2}^!\mathscr{N}_{q_2,k_2}) \\
 {_Hi}_{(1,2)*}\mathscr{H}^{-2}{_Hi}_{(1,2)}^*(\mathscr{M}_{\boldsymbol{q},\boldsymbol{k}}) 
 \ar[u] \ar[d] \\
{_Hi}_{2*}\mathscr{H}^{-1}{_Hi}_{2}^*({_Hi}_{1*}\mathscr{H}^{-1}{_Hi}_{1}^*
(\mathscr{M}_{q_2,k_2})~{_H\otimes}~{_H\pi_1}^!\mathscr{N}_{q_1,k_1}).}
\end{equation}

\end{prop}

\begin{rema}

Here, the functor ${_Hi}_{(1,2)*}\mathscr{H}^{-2}{_Hi}_{(1,2)}^*$ is also defined using definition \ref{defimageinverseH}.

\end{rema}

\begin{proof}

We are going to construct the natural morphism
\[{_Hi}_{(1,2)*}\mathscr{H}^{-2}{_Hi}_{(1,2)}^*(\mathscr{M}_{\boldsymbol{q},\boldsymbol{k}}) \to
{_Hi}_{1*}\mathscr{H}^{-1}{_Hi}_{1}^*({_Hi}_{2*}\mathscr{H}^{-1}{_Hi}_{2}^*
(\mathscr{M}_{q_1,k_1})~{_H\otimes}~{_H\pi_2}^!\mathscr{N}_{q_2,k_2}), 
 \] 
the second one being deduced by symmetry. We will omit the index $H$ in functors to lighten the notation. We denote $U_i:=H_i^c=\{t_i\neq 0\}$ for $i\in\{1,2\}$, we have the natural inclusions
\begin{itemize}
\item $j_1:U_1\hookrightarrow X\times {\Delta}^2$, 
\item $j_2:U_2 \hookrightarrow X\times {\Delta}^2 $,
\item $j_{(1,2)}:U_1\cap U_2=X\times(\Delta^*)^2\hookrightarrow X\times {\Delta}^2$.
\end{itemize}
Following definition \ref{defimageinverseH}, ${i}_{(1,2)*}\mathscr{H}^{-2}{i}_{(1,2)}^*(\mathscr{M}_{\boldsymbol{q},\boldsymbol{k}})$ is the $\mathscr{H}^{-2}$ 
of the simple complex associated to the following double complex:
\begin{equation}\label{doublecomp}\xymatrix{
j_{1!}j_1^{*}\mathscr{M}_{\boldsymbol{q},\boldsymbol{k}} \ar[r] & \overset{\bullet}{\mathscr{M}_{\boldsymbol{q},\boldsymbol{k}}} \\
j_{(1,2)!}j_{(1,2)}^{*}\mathscr{M}_{\boldsymbol{q},\boldsymbol{k}} \ar[r] \ar[u] &
j_{2!}j_2^{*}\mathscr{M}_{\boldsymbol{q},\boldsymbol{k}}. \ar[u]
}\end{equation}
Let us consider the cartesian square
\[\xymatrix{U_1\cap U_2 \ar[r]^{j_2} \ar[d]^{j_1} & U_1 \ar[d]^{j_1} \\
U_2 \ar[r]^{j_2} & X.
}\]
By adjunction, we have a natural transformation of proper base change
\[
{j}_{2!}{j}_{1}^{*}\xrightarrow{}{j}_{1}^{*}{j}_{2!}.
\]
If we apply this morphism to a mixed Hodge module, since proper base change is an isomorphism in the category of perverse sheaves, the perverse sheaves underlying the kernel and the cokernel are zero. Thus, the kernel and cokernel are zero in the category 
$\textup{MHM}(X)$. The natural transformation of proper base change is therefore an isomorphism in $\textup{MHM}(X)$. 

We deduce from that the isomorphism of functors
\[j_{(1,2)!}j_{(1,2)}^{*}\xrightarrow{\sim}j_{1!}j_1^{*}j_{2!}j_2^{*}
\]
using the equality $j_{(1,2)}=j_{2}\circ j_{1}$. We then deduce for the double complex \eqref{doublecomp} the following commutative diagram with short exact sequences:
\[\xymatrix{
i_{1*}\mathscr{H}^{-1}i_1^*\left(\mathscr{M}_{\boldsymbol{q},\boldsymbol{k}}\right) \ar[r] &
j_{1!}j_1^{*}\mathscr{M}_{\boldsymbol{q},\boldsymbol{k}} \ar[r] & \overset{\bullet}{\mathscr{M}_{\boldsymbol{q},\boldsymbol{k}}} \\
i_{1*}\mathscr{H}^{-1}i_1^*\left(j_{2!}j_2^{*}\mathscr{M}_{\boldsymbol{q},\boldsymbol{k}}\right) \ar[r] \ar[u] & 
j_{1!}j_1^{*}j_{2!}j_2^{*}\mathscr{M}_{\boldsymbol{q},\boldsymbol{k}} \ar[r] \ar[u] &
j_{2!}j_2^{*}\mathscr{M}_{\boldsymbol{q},\boldsymbol{k}} \ar[u] \\
i_{1*}\mathscr{H}^{-1}i_1^*\left(i_{2*}\mathscr{H}^{-1}i_2^*\left(\mathscr{M}_{\boldsymbol{q},\boldsymbol{k}}\right)\right). \ar[u]
}\]
We have $i_{1*}\mathscr{H}^{-2}i_1^*=0$, therefore the exactness of the vertical short sequence comes from the same result for perverse sheaves. Since ${i}_{(1,2)*}\mathscr{H}^{-2}{i}_{(1,2)}^*(\mathscr{M}_{\boldsymbol{q},\boldsymbol{k}})$ is the $\mathscr{H}^{-2}$ of the simple complex associated to the double complex
 \eqref{doublecomp}, we have the isomorphism
\begin{equation}\label{isom101}
{i}_{(1,2)*}\mathscr{H}^{-2}{i}_{(1,2)}^*(\mathscr{M}_{\boldsymbol{q},\boldsymbol{k}})\simeq i_{1*}\mathscr{H}^{-1}i_1^*\left(i_{2*}\mathscr{H}^{-1}i_2^*\left(\mathscr{M}_{\boldsymbol{q},\boldsymbol{k}}\right)\right).\end{equation}
The mixed Hodge module ${_H\pi_2}^!\mathscr{N}_{q_2,k_2}$ is flat in the sense of definition \ref{MHMplat}. Indeed, following definition \ref{imageinverse!}, the underlying $\mathscr{D}_{X\times {\Delta}^2}$-module is the inverse image for  $\mathscr{D}$-modules of the module underlying $\mathscr{N}_{q_2,k_2}$. The latter being a free $\mathscr{O}_\Delta$-module, its inverse image by the projection on $\Delta$ is also free hence flat. We then apply lemma \ref{morphismenat} to $i_2$, $\mathscr{M}_{q_1,k_1}$ et ${_H\pi_2}^!\mathscr{N}_{q_2,k_2}$. We get the morphism
\begin{equation}\label{adjonc}{i}_{2*}\mathscr{H}^{-1}{i}_{2}^*(\mathscr{M}_{\boldsymbol{q},\boldsymbol{k}}) \to 
{i}_{2*}\mathscr{H}^{-1}{i}_{2}^*
(\mathscr{M}_{q_1,k_1})~{\otimes}~{\pi_2}^!\mathscr{N}_{q_2,k_2}.
\end{equation}
Combining the morphisms \eqref{isom101} and \eqref{adjonc}, we get the expected morphism.

\end{proof}

By applying the morphism $\mathbf{Nils}$ of definition \ref{Nils} for the functions $t_1:X\times {\Delta}^2\to\mathbb{C}$ and $t_2:X\times {\Delta}^2\to\mathbb{C}$, 
we add two morphisms to diagram \eqref{demidiag} in the following way:

\begin{equation}\label{presquediag}
\xymatrix{\Psi_{t_1}\Psi_{t_2}\mathscr{M} \ar[r] &
{_Hi}_{1*}\mathscr{H}^{-1}{_Hi}_{1}^*\left({_Hi}_{2*}\mathscr{H}^{-1}{_Hi}_{2}^*
(\mathscr{M}_{q_2,k_2})~{_H\otimes}~{_H\pi_2}^!\mathscr{N}_{q_1,k_1}\right) \\
 & {_Hi}_{(1,2)*}\mathscr{H}^{-2}{_Hi}_{(1,2)}^*(\mathscr{M}_{\boldsymbol{q},\boldsymbol{k}}) 
 \ar[u] \ar[d] \\
 \Psi_{t_2}\Psi_{t_1}\mathscr{M} \ar[r] &
{_Hi}_{2*}\mathscr{H}^{-1}{_Hi}_{2}^*\left({_Hi}_{1*}\mathscr{H}^{-1}{_Hi}_{1}^*
(\mathscr{M}_{q_1,k_1})~{_H\otimes}~{_H\pi_1}^!\mathscr{N}_{q_2,k_2}\right).}
\end{equation}

\section{Applying the forgetful functor}

In this section, we consider the projection
 \[\begin{array}{cccc}
    \pi=(\pi_1,\pi_2):& X\times {\Delta}^2 & \to & {\Delta}^2\\
       & (x,t_1,t_2) & \mapsto & (t_1,t_2).
     \end{array}\] 
We denote $\mathbf{H}:=\{H_1,H_2\}$. Let $\mathcal{M}$ be a coherent right $\mathscr{D}_{X\times\Delta^2}$-module such that the pair $(\mathbf{H},\mathcal{M})$ is \emph{without slope}. According to \cite[proposition 2]{Maisonobe}, the canonical $V$-multifiltration satisfies, for all $(\alpha_1,\alpha_2)\in\mathbb{C}^2$,
\begin{equation}\label{commutmaisonobe}\textup{gr}_{\alpha_1}^{V^{H_1}}\textup{gr}_{\alpha_2}^{V^{H_2}}\mathcal{M}
\xleftarrow{\sim}
\textup{gr}_{\alpha_1,\alpha_2}^{V^{\mathbf{H}}}\mathcal{M}
\xrightarrow{\sim}
\textup{gr}_{\alpha_2}^{V^{H_2}}\textup{gr}_{\alpha_1}^{V^{H_1}}\mathcal{M}.
\end{equation}

\begin{defi}

Let $\mathcal{N}_{\alpha,k}$ be the $\mathscr{D}_\Delta$-module underlying the $\tilde{\mathscr{D}}_\Delta$-module of definition \ref{MHMNilss}. Let $\boldsymbol{\alpha}\in[-1,0[^2$ and $\boldsymbol{k}\in \mathbb{N}^2$, we define the following $\mathscr{D}_{\Delta^2}$-module:
\[\mathcal{N}_{\boldsymbol{\alpha},\boldsymbol{k}}:=\mathcal{N}_{\alpha_1,k_1}~{_D\boxtimes}~\mathcal{N}_{\alpha_2,k_2},
\]
for $\mathbf{0}\leq\boldsymbol{\ell}\leq\boldsymbol{k}$, we denote
\[e_{\boldsymbol{\alpha},\boldsymbol{\ell}}:=e_{\alpha_1,\ell_1}\otimes e_{\alpha_2,\ell_2}.\]
We define
\[\mathcal{M}_{\boldsymbol{\alpha},\boldsymbol{k}}:=\mathcal{M}~{_D\otimes}~{_D\pi^*}\mathcal{N}_{\boldsymbol{\alpha},\boldsymbol{k}}.\]

\end{defi}

\begin{prop}\label{VfiltNils}
 
Let $\boldsymbol{\alpha}\in[-1,0[^2$ and $\boldsymbol{k}\in \mathbb{N}^2$. Under the previous hypothesis, the pair $(\mathbf{H}, \mathcal{M}_{\boldsymbol{\alpha},\boldsymbol{k}})$ is without slope and, for all $\boldsymbol{\beta}\in{\mathbb{C}^2}$, we have:
\[V_{\boldsymbol{\beta}}(\mathcal{M}_{\boldsymbol{\alpha},\boldsymbol{k}})=\bigoplus_{\boldsymbol{0}\leq\boldsymbol{\ell}\leq\boldsymbol{k}}
V_{\boldsymbol{\beta}+\boldsymbol{\alpha}+\boldsymbol{1}}
 \left(\mathcal{M}[\frac{1}{t_1t_2}]\right){e}_{\boldsymbol{\alpha},\boldsymbol{\ell}}.
\]

\end{prop}

\begin{proof}

It is the case when $p=2$ of \cite[Proposition 2.26]{kocher1}.

\end{proof}

\begin{defi}\label{DmultNils}

We define the following morphism:

\[\begin{array}{cccc}
\mathbf{{_{D}}N}_{(1,2)}: & \text{gr}^{V^{\mathbf{H}}}_{\boldsymbol{\alpha}}\mathcal{M} & \longrightarrow & \text{gr}^{V^{\mathbf{H}}}_{-1,-1}\mathcal{M}_{\boldsymbol{\alpha},\boldsymbol{k}} \\
& m & \longmapsto & \displaystyle\sum_{\boldsymbol{0}\leq\boldsymbol{\ell}\leq\boldsymbol{k}}
m\cdot(t_1\partial_{t_1}-\alpha_1)^{\ell_1}(t_2\partial_{t_2}-\alpha_2)^{\ell_2}\otimes e_{\boldsymbol{\alpha},\boldsymbol{\ell}}.
\end{array}\]

\end{defi}

\begin{prop}\label{H2tot}

Let us consider the following double complex:

\begin{equation}\label{cplxdouble}\xymatrix{ \textup{gr}^{V^{\mathbf{H}}}_{-1,0}\mathcal{M}_{\boldsymbol{\alpha},\boldsymbol{k}} \ar[r]^{\cdot\partial_{t_1}} &
\textup{gr}^{V^{\mathbf{H}}}_{0,0}\mathcal{M}_{\boldsymbol{\alpha},\boldsymbol{k}} \\
\textup{gr}^{V^{\mathbf{H}}}_{-1,-1}\mathcal{M}_{\boldsymbol{\alpha},\boldsymbol{k}} \ar[r]^{\cdot\partial_{t_1}} \ar[u]^{\cdot\partial_{t_2}} &
\textup{gr}^{V^{\mathbf{H}}}_{0,-1}\mathcal{M}_{\boldsymbol{\alpha},\boldsymbol{k}}  \ar[u]^{\cdot\partial_{t_2}}
}\end{equation}
and Let us denote $\mathbf{Gr_{\bullet}}$ the associated simple complex, concentrated in degree $-2$, $-1$ and $0$. If, for $i=1$ and $i=2$, $k_i$ is bigger than the exponent of the nilpotent endomorphism of right multiplication by $(t_i\partial_{t_i}-\alpha_i)$ on $\textup{gr}^{V^{\mathbf{H}}}_{\boldsymbol{\alpha}}\mathcal{M}$, then the morphism $\mathbf{{_{D}}N}_{(1,2)}$ factorises as an isomorphism
\[\mathbf{{_{D}}N}_{(1,2)}:  \textup{gr}^{V^{\mathbf{H}}}_{\boldsymbol{\alpha}}\mathcal{M}
\xrightarrow{\sim} \mathscr{H}^{-2}(\mathbf{Gr_{\bullet}}).
\]

\end{prop}

\begin{proof}

By definition, $\mathscr{H}^{-2}(\mathbf{Gr_{\bullet}})$ is the intersection of kernels of the right multiplication by $\partial_{t_1}:\textup{gr}^{V^{\mathbf{H}}}_{-1,-1}\mathcal{M}_{\boldsymbol{\alpha},\boldsymbol{k}}\to \textup{gr}^{V^{\mathbf{H}}}_{0,-1}\mathcal{M}_{\boldsymbol{\alpha},\boldsymbol{k}}$ and by
$\partial_{t_2}:\textup{gr}^{V^{\mathbf{H}}}_{-1,-1}\mathcal{M}_{\boldsymbol{\alpha},\boldsymbol{k}}\to \textup{gr}^{V^{\mathbf{H}}}_{-1,0}\mathcal{M}_{\boldsymbol{\alpha},\boldsymbol{k}}$. Furthermore the right multiplication by $t_1$ and $t_2$ on $\mathcal{M}_{\boldsymbol{\alpha},\boldsymbol{k}}$ being bijective, the properties of the canonical $V$-multifiltration imply that the morphisms of right multiplication by
$t_1:\textup{gr}^{V^{\mathbf{H}}}_{0,-1}\mathcal{M}_{\boldsymbol{\alpha},\boldsymbol{k}}\to \textup{gr}^{V^{\mathbf{H}}}_{-1,-1}\mathcal{M}_{\boldsymbol{\alpha},\boldsymbol{k}}$ and by $t_2:\textup{gr}^{V^{\mathbf{H}}}_{-1,0}\mathcal{M}_{\boldsymbol{\alpha},\boldsymbol{k}}\to \textup{gr}^{V^{\mathbf{H}}}_{-1,-1}\mathcal{M}_{\boldsymbol{\alpha},\boldsymbol{k}}$ are bijectives. Therefore we have
\begin{equation}\label{H-2}
\mathscr{H}^{-2}(\mathbf{Gr_{\bullet}})=\bigcap_{1\leq i\leq 2}\ker[\cdot\partial_{t_i}t_i:\textup{gr}^{V^{\mathbf{H}}}_{-1,-1}\mathcal{M}_{\boldsymbol{\alpha},\boldsymbol{k}} \to \textup{gr}^{V^{\mathbf{H}}}_{-1,-1}\mathcal{M}_{\boldsymbol{\alpha},\boldsymbol{k}}].
\end{equation}
According to proposition \ref{VfiltNils}, a section of $\textup{gr}^{V^{\mathbf{H}}}_{-1,-1}\mathcal{M}_{\boldsymbol{\alpha},\boldsymbol{k}}$ can be written in the following way :
\[\sum_{\substack{
0  \leq  \ell_1  \leq  k_1 \\
0 \leq  \ell_2  \leq  k_2
}}
m_{\ell_1,\ell_2}\otimes e_{\boldsymbol{\alpha},\boldsymbol{\ell}}
\]
where the $m_{\ell_1,\ell_2}$ are sections of $\text{gr}^{V^{\mathbf{H}}}_{\boldsymbol{\alpha}}\left(\mathcal{M}[\frac{1}{t_1t_2}]\right)$. According to \eqref{H-2}, this section is in $\mathscr{H}^{-2}(\mathbf{Gr_{\bullet}})$ if and only if for all pairs $(\ell_1,\ell_2)$:
\[m_{\ell_1,\ell_2}\cdot (t_1\partial_{t_1}-\alpha_1)-m_{\ell_1+1,\ell_2}=
m_{\ell_1,\ell_2}\cdot (t_2\partial_{t_2}-\alpha_2)-m_{\ell_1,\ell_2+1}=0,\]
equivalently if and only if for all pair $(\ell_1,\ell_2)$:
\[m_{\ell_1,\ell_2}=m_{0,0}\cdot(t_1\partial_{t_1}-\alpha_1)^{\ell_1}(t_2\partial_{t_2}-\alpha_2)^{\ell_2}.\]
On the other hand, since $\alpha_1<0$ et $\alpha_2<0$, we have $\text{gr}^{V^{\mathbf{H}}}_{\boldsymbol{\alpha}}\left(\mathcal{M}[\frac{1}{t_1t_2}]\right)=\text{gr}^{V^{\mathbf{H}}}_{\boldsymbol{\alpha}}\left(\mathcal{M}\right)$. 
Noticing that, according to the hypothesis made on the size of the $k_i$, for all $m\in\text{gr}^{V^{\mathbf{H}}}_{\boldsymbol{\alpha}}\left(\mathcal{M}\right)$
\[m\cdot(t_1\partial_{t_1}-\alpha_1)^{k_1}=m\cdot(t_2\partial_{t_2}-\alpha_2)^{k_2}=0,\]
we conclude that $\mathbf{{_{D}}N}_{(1,2)}$ factorises as an isomorphism
\[\mathbf{{_{D}}N}_{(1,2)}:  \textup{gr}^{V^{\mathbf{H}}}_{\boldsymbol{\alpha}}\mathcal{M}
\xrightarrow{\sim} \mathscr{H}^{-2}(\mathbf{Gr_{\bullet}}).
\]

\end{proof}

\begin{lemm}\label{isomnat}

We have a natural isomorphism
\[{_Di}_{(1,2)*}\mathscr{H}^{-2}{_Di}_{(1,2)}^!(\mathcal{M}_{\boldsymbol{\alpha},\boldsymbol{k}}) 
\xrightarrow{\sim}
\mathscr{H}^{-2}(\mathbf{Gr_{\bullet}}).
\]

\end{lemm}

\begin{proof}

Following the proof of proposition \ref{constructiondemidiag}, we notice the following natural isomorphisms
\[{_Di}_{2*}\mathscr{H}^{-1}{_Di}_{2}^!{_Di}_{1*}\mathscr{H}^{-1}{_Di}_{1}^!(\mathcal{M}_{\boldsymbol{\alpha},\boldsymbol{k}})\xleftarrow{\sim}{_Di}_{(1,2)*}\mathscr{H}^{-2}{_Di}_{(1,2)}^!(\mathcal{M}_{\boldsymbol{\alpha},\boldsymbol{k}}) 
\xrightarrow{\sim} {_Di}_{1*}\mathscr{H}^{-1}{_Di}_{1}^!{_Di}_{2*}\mathscr{H}^{-1}{_Di}_{2}^!(\mathcal{M}_{\boldsymbol{\alpha},\boldsymbol{k}})
\]

In the category of $\mathscr{D}$-modules, we still have the exact sequence  \eqref{SEsaito} (cf \cite[(2.24.3)]{SaitoMHM}). We apply the sequence twice to the double complex \eqref{cplxdouble} and, thanks to the exactness of the nearby cycles functors, we get the following diagram of exact sequences:
\[\xymatrix{
0 \ar[r] &
\textup{gr}_0^{V^{H^2}}\left({_Di}_{1*}\mathscr{H}^{-1}{_Di}_{1}^!\mathcal{M}_{\boldsymbol{\alpha},\boldsymbol{k}}\right) \ar[r] &
 \textup{gr}^{V^{\mathbf{H}}}_{-1,0}\mathcal{M}_{\boldsymbol{\alpha},\boldsymbol{k}} \ar[r] &
\textup{gr}^{V^{\mathbf{H}}}_{0,0}\mathcal{M}_{\boldsymbol{\alpha},\boldsymbol{k}} \\
0 \ar[r] &
\textup{gr}_{-1}^{V^{H^2}}\left({_Di}_{1*}\mathscr{H}^{-1}{_Di}_{1}^!\mathcal{M}_{\boldsymbol{\alpha},\boldsymbol{k}}\right) \ar[u] \ar[r] &
\textup{gr}^{V^{\mathbf{H}}}_{-1,-1}\mathcal{M}_{\boldsymbol{\alpha},\boldsymbol{k}} \ar[r] \ar[u] &
\textup{gr}^{V^{\mathbf{H}}}_{0,-1}\mathcal{M}_{\boldsymbol{\alpha},\boldsymbol{k}}  \ar[u] \\
0 \ar[r] &
{_Di}_{(1,2)*}\mathscr{H}^{-2}{_Di}_{(1,2)}^!(\mathcal{M}_{\boldsymbol{\alpha},\boldsymbol{k}}) \ar[u] \ar[r] &
\textup{gr}_{-1}^{V^{H^1}}\left({_Di}_{2*}\mathscr{H}^{-1}{_Di}_{2}^!\mathcal{M}_{\boldsymbol{\alpha},\boldsymbol{k}}\right) \ar[r] \ar[u] &
\textup{gr}_0^{V^{H^1}}\left({_Di}_{2*}\mathscr{H}^{-1}{_Di}_{2}^!\mathcal{M}_{\boldsymbol{\alpha},\boldsymbol{k}}\right) \ar[u] \\
& 0 \ar[u] & 0 \ar[u] & 0. \ar[u] &
}\]
the rows and the columns of this diagram are exact, thus we have the isomorphism
\[{_Di}_{(1,2)*}\mathscr{H}^{-2}{_Di}_{(1,2)}^!(\mathcal{M}_{\boldsymbol{\alpha},\boldsymbol{k}}) 
\xrightarrow{\sim}
\mathscr{H}^{-2}(\mathbf{Gr_{\bullet}}).
\]

\end{proof}

We deduce from lemma \ref{isomnat} and proposition \ref{H2tot} an isimorphism
\[\text{gr}^{V^{\mathbf{H}}}_{\boldsymbol{\alpha}}\mathcal{M} \xrightarrow{\sim} {_Di}_{(1,2)*}\mathscr{H}^{-2}{_Di}_{(1,2)}^!(\mathcal{M}_{\boldsymbol{\alpha},\boldsymbol{k}}).\]
We combine this with mophisms \eqref{commutmaisonobe} and diagram \eqref{presquediag}
to which we apply the forgetful functor from the category of mixed Hodge modules to the category of $\mathscr{D}_X$-modules and we get yhe following commutative diagram:
\begin{equation}\label{diag}
\xymatrix{\textup{gr}_{\alpha_1}^{V^{H_1}}\textup{gr}_{\alpha_2}^{V^{H_2}}\mathcal{M} \ar[r]^-{(a)} &
{_Di}_{1*}\mathscr{H}^{-1}{_Di}_{1}^!\left({_Di}_{2*}\mathscr{H}^{-1}{_Di}_{2}^!
(\mathcal{M}_{\alpha_2,k_2})~{_D\otimes}~{_D\pi_1}^*\mathcal{N}_{\alpha_1,k_1}\right) \\
\textup{gr}_{\boldsymbol{\alpha}}^{V^{\mathbf{H}}}\mathcal{M} \ar[r]^-{\simeq} \ar[d] \ar[u]^-{\simeq}
 & {_Di}_{(1,2)*}\mathscr{H}^{-2}{_Di}_{(1,2)}^!(\mathcal{M}_{\boldsymbol{\alpha},\boldsymbol{k}})
 \ar[u] \ar[d] \\
\textup{gr}_{\alpha_2}^{V^{H_2}}\textup{gr}_{\alpha_1}^{V^{H_1}}\mathcal{M} \ar[r] &
{_Di}_{2*}\mathscr{H}^{-1}{_Di}_{2}^!\left({_Di}_{1*}\mathscr{H}^{-1}{_Di}_{1}^!
(\mathcal{M}_{\alpha_1,k_1})~{_D\otimes}~{_D\pi_2}^*\mathcal{N}_{\alpha_2,k_2}\right).}
\end{equation}

\begin{prop}\label{aiso}

Morphism $(a)$ is an isomorphism.

\end{prop}

\begin{proof}

We show, in the same way as for proposition \ref{H2tot}, that for $j\in\{1,2\}$
\[\textup{gr}_{\alpha_j}^{V^{H_j}}\mathcal{M}\simeq {_Di}_{j*}\mathscr{H}^{-1}{_Di}_{j}^!
(\mathcal{M}_{\alpha_j,k_j})\]
when $k_1$ and $k_2$ are big enough. To conclude, it is then enough to apply this isomorphism twice.

\end{proof}

\section{Commutativity theorem}

\begin{theo}\label{commut2MHM}

Let $\mathscr{M}\in \textup{MHM}(X\times{\Delta}^2)$ be a mixed Hodge module on $X\times{\Delta}^2$, such that the pair $(\mathbf{H},\mathcal{M})$ is without slope where $\mathcal{M}$ is the right $\mathscr{D}_{X\times\Delta^2}$-module underlying $\mathscr{M}$. We then have an isomorphism
\[\Psi_{t_1}\Psi_{t_2}\mathscr{M} \simeq \Psi_{t_2}\Psi_{t_1}\mathscr{M}.
\]

\end{theo}

\begin{proof}

Let us consider the diagram \eqref{presquediag}
\[
\xymatrix{\Psi_{t_1}\Psi_{t_2}\mathscr{M} \ar[r] &
{_Hi}_{1*}\mathscr{H}^{-1}{_Hi}_{1}^*\left({_Hi}_{2*}\mathscr{H}^{-1}{_Hi}_{2}^*
(\mathscr{M}_{q_2,k_2})~{_H\otimes}~{_H\pi_2}^!\mathscr{N}_{q_1,k_1}\right) \\
 & {_Hi}_{(1,2)*}\mathscr{H}^{-2}{_Hi}_{(1,2)}^*(\mathscr{M}_{\boldsymbol{q},\boldsymbol{k}}) 
 \ar[u] \ar[d] \\
 \Psi_{t_2}\Psi_{t_1}\mathscr{M} \ar[r] &
{_Hi}_{2*}\mathscr{H}^{-1}{_Hi}_{2}^*\left({_Hi}_{1*}\mathscr{H}^{-1}{_Hi}_{1}^*
(\mathscr{M}_{q_1,k_1})~{_H\otimes}~{_H\pi_1}^!\mathscr{N}_{q_2,k_2}\right).}
\]
We are going to show that under the hypothesis of the theorem, every arrow of this diagram is an isomorphism. When we apply the forgetful functor from the category of mixed Hodge modules to the category of right $\mathscr{D}_X$-modules, the diagram \eqref{diag} and proposition \ref{aiso} imply that the arrows of diagram \eqref{presquediag} are isomorphisms in the category of right $\mathscr{D}_X$-modules. Yet, according to proposition \ref{isoDH}, if a morphism of mixed Hodge modules is an isomorphism between the underlying right $\mathscr{D}_X$-modules then it is an isomorphism of mixed Hodge modules. Thus, we get an isomorphism
\[\Psi_{t_1}\Psi_{t_2}\mathscr{M} \simeq \Psi_{t_2}\Psi_{t_1}\mathscr{M}.
\]
\end{proof}

Let us consider the projection 
 \[\begin{array}{cccc}
    \pi=(\pi_1,\dots ,\pi_p):& X\times {\Delta}^p & \to & {\Delta}^p\\
       & (x,t_1,\dots ,t_p) & \mapsto & (t_1,\dots ,t_p)
     \end{array}\] 
and $\mathbf{H}:=\{H_1,\dots ,H_p\}$ where $H_i:=\{t_i=0\}$. We deduce from theorem \ref{commut2MHM} the following corollary.

\begin{coro}\label{theocommuthodge}

Let $\mathscr{M}\in \textup{MHM}(X\times{\Delta}^p)$ be a mixed Hodge module on $X\times{\Delta}^p$ such that the pair $(\mathbf{H},\mathcal{M})$ is without slope where $\mathcal{M}$ is the right $\mathscr{D}_{X\times\Delta^p}$-module underlying $\mathscr{M}$. Then, for all permutation $\sigma$ of $\{1,\dots ,p\}$, we have an isomorphism
\[\Psi_{t_1}\cdots \Psi_{t_p}\mathscr{M} \simeq \Psi_{t_{\sigma(1)}}\cdots \Psi_{t_{\sigma(p)}}\mathscr{M}.
\]

\end{coro}

\begin{proof}

The proposition 3 of \cite{Maisonobe} asserts that if $(\mathbf{H},\mathcal{M})$ is without slope
then, for all $I\subset\{1,\dots ,p\}$, the pair $(\mathbf{H}_{I^c},\textup{gr}_{\boldsymbol{\alpha}_I}^{V^{\mathbf{H}_I}}\mathcal{M})$ is without slope. 
We get can therefore apply as many times as necessary theorem \ref{commut2MHM} to the pairs of consecutive functors $\Psi_{t_i}$. Thus we get the expected isomorphism.

\end{proof}

\part{Mixed Hodge modules without slope}\label{section2}

In this part, we define a condition analogous to the condition without slope for \emph{filtered} $\mathscr{D}$-modules. It will be called strict $\mathbb{R}$-multispecialisability. Actually, we will rather consider the associated Rees's module (cf. \ref{Rees}). Thus the strict $\mathbb{R}$-multispecialisability will be defined for $R_F\mathscr{D}$-modules.

In section 8.1, we define the strict $\mathbb{R}$-multispecialisability. Then we determine the connection between the filtered $\mathscr{D}$-modules case and the $R_F\mathscr{D}$-modules case when they underlie a mixed Hodge module. We will see that the notion of compatibility of the filtrations plays a key role (cf. \ref{Fcomp}).

In section 8.2, we show that strict $\mathbb{R}$-multispecialisability is preserved by a proper direct image. This result is analogous to the direct image theorem of M. Saito for mixed Hodge modules. However our approach is different from M. Saito's who considers filtered $\mathscr{D}$-modules and not $R_F\mathscr{D}$-modules. This enables us to work in an \emph{abelian} category and thus to handle more easily the direct image functor in the corresponding derived category.

In section 8.3, we study the behaviour of the weight filtrations. We expect this behaviour to be the same as with polarizable variations of Hodge structures in the neighbourhood of a normal crossing divisor. We prove that this condition is preserved by a projective direct image.

In section 8.4, we are interested in the case of a quasi-ordinary hypersurfaces singularities. Some pure Hodge modules with support on such singularities can be naturally seen as the direct image of variations of Hodge structures defined outside of a normal crossing divisor. Thus we will be able to apply the previous results.

\section{Strict $\mathbb{R}$-multispecialisability}

Here, we define the strict $\mathbb{R}$-multispecialisability and we determine the connection between this condition and the compatibility of the filtrations in the sense of definition \ref{Fcomp}.

\begin{defi}[multispecialisability]

Let $\mathbf{H}:=\{H_1,\dots ,H_p\}$ $p$ be smooth hypersurfaces of $X$ and let $\mathfrak{M}$ be a (right) coherent $\tilde{\mathscr{D}}_X$-module. 
\begin{enumerate}
\item $\mathfrak{M}$ is said to be \emph{multispecialisable} along $\mathbf{H}$ if, in the neighbourhood of any point of $X$, there exists a good $V$-multifiltration $U_\bullet\mathfrak{M}$ of $\mathfrak{M}$, polynomials $b_i(s)=\prod(s-\alpha_iz)$ (for $1\leq i\leq p$) and integers $\ell_i\in\mathbb{N}$ (for $1\leq i\leq p$) such that, for all $\boldsymbol{k}\in\mathbb{Z}^p$
\begin{equation}\label{eqbern1}
U_{\boldsymbol{k}}\mathfrak{M}\cdot z^{\ell_i}b_i(E_i-k_iz)\subset U_{\boldsymbol{k-1}_i}\mathfrak{M}.
\end{equation}

\item $\mathfrak{M}$ is said to be \emph{multispecialisable by section} along $\mathbf{H}$ if, for any local section $m$ of $\mathfrak{M}$ and for all $1\leq i\leq p$, there exists $\ell_i\in\mathbb{N}$, $b_i(s)=\prod(s-\alpha_iz)$ and $P_i\in V_{\mathbf{-1}_i}\tilde{\mathscr{D}}_X$ such that
\begin{equation}\label{eqbern2}
m\cdot (z^{\ell_i}b_i(E_i)-P_i) = 0.
\end{equation}

\end{enumerate}

\end{defi}

\begin{rema}

\begin{itemize}
\item These two definitions are equivalent, and if they are satisfied, then condition 1. is satisfied for any good $V$-multifiltration.
\item If $\mathfrak{M}$ is multispecialisable along $\mathbf{H}$, the unitary polynomials of smallest degree satisfying equation \eqref{eqbern1} or equation \eqref{eqbern2} are called \emph{weak Bernstein-Sato polynomials}.
\item If these properties are satisfied and if the roots of the polynomials $b_i$ are real numbers, then $\mathfrak{M}$ is said to be $\mathbb{R}$-{multispecialisable} along $\mathbf{H}$.
\item In the case $\mathbb{R}$-multispecialisable by section, we can define the order multifiltration along $\mathbf{H}$ if we consider the unitary polynomials of smallest degree satisfying equation \eqref{eqbern2} for a local section $m$.

\end{itemize}

\end{rema}

\begin{defi}

let $\mathcal{M}$  $\mathscr{D}_X$-module such that the pair $(\mathbf{H},\mathcal{M})$ is without slope and such that the roots of the Bernstein-Sato polynomials are real numbers. We will say as well that $\mathcal{M}$ is $\mathbb{R}$-multispecialisable along $\mathbf{H}$.

\end{defi}

\begin{rema}\label{multispeF}

Let $\mathcal{M}$ be a coherent $\mathscr{D}_X$-module with a coherent $F$-filtration. If $\mathcal{M}$ is $\mathbb{R}$-multispecialisable along $\mathbf{H}$, then the Rees's module $R_F\mathcal{M}$ is $\mathbb{R}$-multispecialisable along $\mathbf{H}$ and the converse is also true.

\end{rema}

\begin{defi}[strict $\mathbb{R}$-multispecialisability]\label{defmultstrict}

Let $\mathfrak{M}$ be a $\tilde{\mathscr{D}}_X$-module $\mathbb{R}$-multispecialisable along $\mathbf{H}$. $\mathfrak{M}$ is said to be \emph{strictly} $\mathbb{R}$-multispecialisable along $\mathbf{H}$, if

\begin{enumerate}

\item There exists a finite set $A\in ]-1,0]^p$ such that the order multifiltration along $\mathbf{H}$ is a good $V$-multifiltration indexed by $A+\mathbb{Z}^p$,
\item for all $I\subset\{1,\dots ,p\}$ and all $\boldsymbol{\alpha}\in \mathbb{R}^p$, $V_{\boldsymbol{\alpha}}\mathfrak{M}/V_{<\boldsymbol{\alpha}_I,\boldsymbol{\alpha}_{I^c}}\mathfrak{M}$ is strict,
\item for all $1\leq i\leq p$ and all $\alpha_i<0$, the morphism $t_i:\textup{gr}_{\boldsymbol{\alpha}}\mathfrak{M}\to \textup{gr}_{\boldsymbol{\alpha-1}_i}\mathfrak{M}$ is surjective,
\item for all $1\leq i\leq p$ and all $\alpha_i>-1$, the morphism $\eth_i:\textup{gr}_{\boldsymbol{\alpha}}\mathfrak{M}\to \textup{gr}_{\boldsymbol{\alpha+1}_i}\mathfrak{M}(-1)$ is surjective.

\end{enumerate}

In this case, we call the order $V$-multifiltration \emph{the canonical $V$-multifiltration}.

\end{defi}

\begin{rema}
 If $\mathfrak{M}$ is \emph{strictly} $\mathbb{R}$-multispecialisable along $\mathbf{H}$, then
\begin{itemize}
\item for all $1\leq i\leq p$ and all $\alpha_i<0$, the morphism $t_i:\textup{gr}_{\boldsymbol{\alpha}}\mathfrak{M}\to \textup{gr}_{\boldsymbol{\alpha-1}_i}\mathfrak{M}$ is an isomorphism,
\item for all $1\leq i\leq p$ and all $\alpha_i>-1$, the morphism $\eth_i:\textup{gr}_{\boldsymbol{\alpha}}\mathfrak{M}\to \textup{gr}_{\boldsymbol{\alpha+1}_i}\mathfrak{M}(-1)$ is an isomorphism.
\end{itemize}

\end{rema}

\begin{prop}\label{compatstrict}

Let $\mathcal{M}$ be a coherent $\mathscr{D}_X$-module, $\mathbb{R}$-{multispecialisable} along $\mathbf{H}$ and equipped with a coherent $F$-filtration such that
\begin{enumerate}

\item The Rees module $R_F\mathcal{M}$ underlies a mixed Hodge module,
\item The filtrations $(F_\bullet\mathcal{M},V^1_\bullet\mathcal{M},\dots ,V^p_\bullet\mathcal{M})$ are compatible.

\end{enumerate}

Then $R_F\mathcal{M}$ is strictly $\mathbb{R}$-multispecialisable along $\mathbf{H}$.

\end{prop}

\begin{proof}

According to Remark \ref{multispeF}, $R_F\mathcal{M}$ is $\mathbb{R}$-multispecialisable along $\mathbf{H}$. We define a $V$-multifiltration of $R_F\mathcal{M}$ in the following way:
\begin{equation}\label{multifilt}
U_{\boldsymbol{\alpha}}R_F\mathcal{M}:=\bigoplus_{p\in\mathbb{Z}}(F_pV_{\boldsymbol{\alpha}}\mathcal{M})z^p.
\end{equation}

By definition, there exists a finite set $A\subset ]-1,0]^p$ such that this $V$-multifiltration is indexed by $A+\mathbb{Z}^p$. The filtrations $(F_\bullet\mathcal{M},V^1_\bullet\mathcal{M},\dots ,V^p_\bullet\mathcal{M})$ being compatible, we get that, for all $I\subset\{1,\dots ,p\}$ and all $\boldsymbol{\alpha}\in \mathbb{R}^p$,  $U_{\boldsymbol{\alpha}}R_F\mathcal{M}/U_{<\boldsymbol{\alpha}_I,\boldsymbol{\alpha}_{I^c}}R_F\mathcal{M}=
R_F\left(V_{\boldsymbol{\alpha}}\mathcal{M}/V_{<\boldsymbol{\alpha}_I,\boldsymbol{\alpha}_{I^c}}\mathcal{M}\right)$,
thus $U_{\boldsymbol{\alpha}}R_F\mathcal{M}/U_{<\boldsymbol{\alpha}_I,\boldsymbol{\alpha}_{I^c}}R_F\mathcal{M}$ is strict.

 Let $mz^p$ be a local section of $(F_pV_{\boldsymbol{\alpha}}\mathcal{M})z^p$, for all $1\leq i\leq p$, there exists an integer $\nu_i$ such that 
\[mz^p(t_i\eth_i-\alpha_iz)^{\nu_i}\in (F_{p+\nu_i}V_{\alpha_1,\dots ,<\alpha_i,\dots ,\alpha_p}\mathcal{M})z^{p+\nu_i},\]
thus, we have $U_{\boldsymbol{\alpha}}R_F\mathcal{M}\subset V_{\boldsymbol{\alpha}}R_F\mathcal{M}$.

 We are going to show that the $V$-multifiltration $U_{\bullet}R_F\mathcal{M}$ is a good $V$-multifiltration, this will imply that it is the canonical $V$-multifiltration and therefore that points 1. and 2. of definition \ref{defmultstrict} are satisfied. We first prove points 3. and 4. of definition \ref{defmultstrict}. Since The Rees's module $R_F\mathcal{M}$ underlies a mixed Hodge module, it is strictly $\mathbb{R}$-specialisable along any hypersurface, in particular we have for all $1\leq i\leq p$:
\begin{enumerate}[label=(\alph*')]

\item $\forall p$ and $\forall \alpha_i<0, t_i:F_pV_{\alpha_i}^{H_i}\mathcal{M}\to
F_pV_{\alpha_i-1}^{H_i}\mathcal{M}$ is an isomorphism,
\item\label{b'} $\forall p$ and $\forall \alpha_i>-1, \partial_i:F_p\textup{gr}_{\alpha_i}^{V_i}\mathcal{M}\to
F_{p+1}\textup{gr}_{\alpha_i+1}^{V_i}\mathcal{M}$ is an isomorphism.

\end{enumerate}

Furthermore, we deduce from the $\mathbb{R}$-multispecialisability property of $\mathcal{M}$ that, for all $1\leq i\leq p$:

\begin{enumerate}[label=(\alph*'')]
\item $\forall \alpha_i<0, t_i:V_{\boldsymbol{\alpha}}\mathcal{M}\to
V_{\boldsymbol{\alpha-1}_i}\mathcal{M}$ is an isomorphism,
\item\label{b''} $\forall \alpha_i>-1, \partial_i:\textup{gr}_{\boldsymbol{\alpha}}\mathcal{M}\to
\textup{gr}_{\boldsymbol{\alpha+1}_i}\mathcal{M}$ is an isomorphism.

\end{enumerate}
We can then deduce that for all $1\leq i\leq p$:
\begin{enumerate}[label=(\alph*)]

\item\label{a}
 $\forall p$ and $\forall \alpha_i<0, t_i:F_pV_{\boldsymbol{\alpha}}\mathcal{M}\to
F_pV_{\boldsymbol{\alpha-1}_i}\mathcal{M}$ is an isomorphism,
\item\label{b} $\forall p$ and $\forall \alpha_i>-1, \partial_i:F_p\textup{gr}_{\boldsymbol{\alpha}}\mathcal{M}\to
F_{p+1}\textup{gr}_{\boldsymbol{\alpha+1}_i}\mathcal{M}$ is an isomorphism.

\end{enumerate}

let us prove \ref{b}, in order to lighten notation, we set $i=1$. Injectivity is a direct consequence of \ref{b''}. For surjectivity, we use the compatibility of filtrations. Let $\bar{m}\in F_{p+1}\textup{gr}_{\boldsymbol{\alpha+1}_1}\mathcal{M}$ and $m\in F_{p+1}V_{\alpha_2,\dots ,\alpha_p}\textup{gr}^{V_1}_{\alpha_1+1}\mathcal{M}$ be a representative. According to the surjectivity of \ref{b'}, there exists $\tilde{m}\in F_p\textup{gr}^{V_1}_{\alpha_1}\mathcal{M}$ such that $\partial_1\tilde{m}=m$. According to the surjectivity of \ref{b''}, there exists $m'\in V_{\alpha_2,\dots ,\alpha_p}\textup{gr}^{V_1}_{\alpha_1}\mathcal{M}$ and $n\in\sum_{j=2}^pV_{\alpha_2
,\dots ,<\alpha_j,\dots ,\alpha_p}\textup{gr}^{V_1}_{\alpha_1+1}\mathcal{M}$ such that $\partial_1 m'=n+m$. For $2\leq i\leq p$, there exists real numbers $\beta_i$ such that $\beta_i\geq \alpha_i$ and $\tilde{m}-m'\in V_{\beta_2
,\dots ,\beta_p}\textup{gr}^{V_1}_{\alpha_1}\mathcal{M}$. Following \ref{b''}, the morphism $\partial_1:\textup{gr}^{V_2}_{\beta_2}\textup{gr}_{\alpha_1,\beta_3\dots ,\beta_p}\mathcal{M}\to
\textup{gr}^{V_2}_{\beta_2}\textup{gr}_{\alpha_1+1,\beta_3\dots ,\beta_p}\mathcal{M}$ is injective. Yet, the class of $n$ in $\textup{gr}^{V_2}_{\beta_2}\textup{gr}_{\alpha_1+1,\beta_3\dots ,\beta_p}\mathcal{M}$ is zero and $\partial_1(\tilde{m}-m')=n$. we deduce that $\tilde{m}-m'\in V_{<\beta_2
,\dots ,\beta_p}\textup{gr}^{V_1}_{\alpha_1}\mathcal{M}$. Thus, we show by induction that
$\tilde{m}-m'\in V_{\alpha_2
,\dots ,\alpha_p}\textup{gr}^{V_1}_{\alpha_1}\mathcal{M}$. Finally, $\tilde{m}\in F_pV_{\alpha_2
,\dots ,\alpha_p}\textup{gr}^{V_1}_{\alpha_1}\mathcal{M}$ et $\partial_1\tilde{m}=m$, thus we get the surjectivity in \ref{b}. Statement \ref{a} can be proven in the same way.

This implies points 3. and 4. in definition \ref{defmultstrict} for the $V$-multifiltration $U_{\bullet}R_F\mathcal{M}$. 

Thus it remains to prove that it is a good $V$-multifiltration. The set $A$ of indices in 
$[-1,0]^p$ of the multifiltration $U_\bullet R_F\mathcal{M}$ being finite, in order to show coherence it is enough to prove that for all $\boldsymbol{\alpha}\in [-1,0]^p$, the module $U_{\boldsymbol{\alpha}}R_F \mathcal{M}$ is $V_{\mathbf{0}}R_F\mathscr{D}_X$-coherent, and we will conclude by using the two isomorphisms above. More precisely we are going to show that $U_{\boldsymbol{\alpha}}R_F \mathcal{M}$ is $R_F\mathscr{D}_{X/\mathbb{C}^p}$-coherent where we consider the reduced local equations $(t_1,\dots ,t_p):X\to \mathbb{C}^p$ of $\mathbf{H}$.

let us show that $F_pV_{\boldsymbol{\alpha}}\mathcal{M}$ is $\mathcal{O}_X$-coherent for all $p$. The module $\mathcal{M}$ is $\mathbb{R}$-{multispecialisable} along $\mathbf{H}$ therefore 
$V_{\boldsymbol{\alpha}}\mathcal{M}$ is $V_{\boldsymbol{0}}\mathscr{D}_X$-coherent, we can thus filter $V_{\boldsymbol{\alpha}}\mathcal{M}$ by $\mathcal{O}_X$-coherent submodules $(V_{\boldsymbol{\alpha}}\mathcal{M})_l$. The increasing sequence of coherent
 $\mathcal{O}_X$-modules $F_p\mathcal{M}\cap (V_{\boldsymbol{\alpha}}\mathcal{M})_l$ is included in $F_p\mathcal{M}$ therefore it is locally stationary, thus $F_pV_{\boldsymbol{\alpha}}\mathcal{M}$ is $\mathcal{O}_X$-coherent.

It remains to show that, locally, there exists $p_0$ such that $F_p\mathscr{D}_{X/\mathbb{C}^p}\cdot F_{p_0}V_{\boldsymbol{\alpha}}\mathcal{M}=F_{p+p_0}V_{\boldsymbol{\alpha}}\mathcal{M}$ for all $p\geq 0$. The Rees's module $R_F\mathcal{M}$ underlies a mixed Hodge module and the filtrations $(F_\bullet\mathcal{M},V^1_\bullet\mathcal{M},\dots ,V^p_\bullet\mathcal{M})$ are compatible, therefore $R_F\textup{gr}_{\boldsymbol{\alpha}}\mathcal{M}$ also underlies a mixed Hodge module on $\cap_iH_i$ and thus it is $R_F\mathscr{D}_{\cap_iH_i}$-coherent. We deduce from that the $R_F\mathscr{D}_{\cap_iH_i}$-coherence of
\[\frac{R_FV_{\boldsymbol{\alpha}}\mathcal{M}}{\displaystyle\sum_{1\leq i\leq p}R_FV_{\boldsymbol{\alpha-1}_i}\mathcal{M}}
\]
and therefore the existence of $p_0$ such that, for all $p\geq 0$,
\begin{equation}\label{egalitécohérence}
F_p\mathscr{D}_{\cap_iH_i}\cdot 
\frac{F_{p_0}V_{\boldsymbol{\alpha}}\mathcal{M}}{\displaystyle\sum_{1\leq i\leq p}F_{p_0}V_{\boldsymbol{\alpha-1}_i}\mathcal{M}}=\frac{F_{p+p_0}V_{\boldsymbol{\alpha}}\mathcal{M}}{\displaystyle\sum_{1\leq i\leq p}F_{p+p_0}V_{\boldsymbol{\alpha-1}_i}\mathcal{M}}.
\end{equation}

let us denote $U_{\boldsymbol{\alpha},p}:=F_p\mathscr{D}_{X/\mathbb{C}^p}\cdot 
F_{p_0}V_{\boldsymbol{\alpha}}\mathcal{M}$. according to \ref{a}, if $\boldsymbol{\alpha}\in [-1,0[^p$, then we can rewrite the equality \eqref{egalitécohérence}:
\[
\frac{U_{\boldsymbol{\alpha},p}}{\displaystyle\sum_{1\leq i\leq p}U_{\boldsymbol{\alpha},p}\cdot t_i} =
\frac{F_{p+p_0}V_{\boldsymbol{\alpha}}\mathcal{M}}{\displaystyle\sum_{1\leq i\leq p}F_{p+p_0}V_{\boldsymbol{\alpha}}\mathcal{M}\cdot t_i}.
\]
According to Nakayama's lemma, this implies that $U_{\boldsymbol{\alpha},p}=F_{p+p_0}V_{\boldsymbol{\alpha}}\mathcal{M}$ in the neighbourhood of $\cap_i H_i$. 

There remains the case where $\alpha_i=0$ for some index $1\leq i\leq p$. We use an induction on the number of indices  $1\leq i\leq $ such that $\alpha_i=0$. In order to simplify, let us suppose that there exists $1\leq k\leq p$ such that $(\alpha_1,\dots ,\alpha_k)\in  
[-1,0[^k$ and $\alpha_{k+1}=\dots =\alpha_p=0$. We can apply the above argument to $R_F\textup{gr}_{\alpha_{k+1}\dots \alpha_p}^{H_{k+1}\dots H_{p}}\mathcal{M}$ and deduce that
\[F_p\mathscr{D}_{(\bigcap_{i>k}H_i)/\mathbb{C}^k}F_{p_0}V_{\alpha_1\dots \alpha_k}\textup{gr}_{\alpha_{k+1}\dots \alpha_p}\mathcal{M}=F_{p+p_0}V_{\alpha_1\dots \alpha_k}\textup{gr}_{\alpha_{k+1}\dots \alpha_p}\mathcal{M}.\]
We can rewrite this equality
\[
F_{p+p_0}V_{\boldsymbol{\alpha}}\mathcal{M}=F_p\mathscr{D}_{X/\mathbb{C}^p}F_{p_0}
V_{\boldsymbol{\alpha}}\mathcal{M}+\displaystyle\sum_{k+1\leq i\leq p}F_{p+p_0}V_{\alpha_1\dots <\alpha_i\dots \alpha_p}\mathcal{M}.
\]

In the sum on the right side, we decreased the number of indices $1\leq i\leq $ such that $\alpha_i=0$ by $1$. Thus, using the induction hypothesis, we have \[F_{p+p_0}V_{\alpha_1\dots <\alpha_i\dots \alpha_p}\mathcal{M}=F_p\mathscr{D}_{X/\mathbb{C}^p}F_{p_0}
V_{\alpha_1\dots <\alpha_i\dots \alpha_p}\mathcal{M}\]
and we deduce that 
\[F_{p+p_0}V_{\boldsymbol{\alpha}}\mathcal{M}=F_p\mathscr{D}_{X/\mathbb{C}^p}F_{p_0}
V_{\boldsymbol{\alpha}}\mathcal{M}.\]

The multifiltration \eqref{multifilt} is therefore a good $V$-multifiltration of $R_F\mathcal{M}$, then it is the canonical $V$-multifiltration, which gives us points 
 1. and 2. of definition \ref{defmultstrict} and concludes the proof of the proposition.

\end{proof}

Conversely, strict $\mathbb{R}$-multispecialisability implies the compatibility of the filtrations.

\begin{prop}\label{strictmulcompat}

If $R_F\mathcal{M}$ is strictly $\mathbb{R}$-multispecialisable along $\mathbf{H}$, then the filtrations $(F_\bullet\mathcal{M},V^1_\bullet\mathcal{M},\dots ,V^p_\bullet\mathcal{M})$ are compatible. 

\end{prop}

\begin{proof}

In order to prove this proposition, we denote $\mathscr{M}$ the $\mathbb{C}[z,y_1,\dots ,y_p]$-module given by the Rees's construction of definition \ref{Rees} starting with $\mathcal{M}$, the filtration $F_\bullet\mathcal{M}$ and the canonical $V$-filtrations $V^{H_i}_\bullet\mathcal{M}$. In order to show the compatibility of the filtrations, we will show that $\mathscr{M}$ is a flat $\mathbb{C}[z,y_1,\dots ,y_p]$-module (proposition \ref{compatReesplat}). According to corollary \ref{coroKoszul}, it is enough to show that, for all $0\leq k\leq p$, the sequence $y_1,\dots ,y_k,z$ is regular.

Let us first prove that multiplication by $y_{\ell}$ on $\mathscr{M}/(y_1,\dots ,y_{\ell-1})\mathscr{M}$ is injective for all $1\leq \ell\leq k$. let us consider the following diagram:
\[\xymatrix{\mathscr{M}/(y_1,\dots ,y_{\ell-1})\mathscr{M} \ar[r]^{.y_{\ell}} \ar[d] & 
\mathscr{M}/(y_1,\dots ,y_{\ell-1})\mathscr{M} \ar[d] \\
\mathscr{M}/(y_1,\dots ,y_{\ell-1},z-1)\mathscr{M} \ar[r]^{.y_\ell}  & 
\mathscr{M}/(y_1,\dots ,y_{\ell-1},z-1)\mathscr{M}.}\]
According to lemma \ref{multispeF}, $\mathcal{M}$ is $\mathbb{R}$-multispecialisable along $\mathbf{H}$. Therefore, according to \cite[proposition 2.12]{kocher1}, the filtrations $(V^{H_1}_\bullet\mathcal{M},\dots ,V^{H_p}_\bullet\mathcal{M})$ of $\mathcal{M}$ are compatible. We deduce that the lower horizontal morphism in the above diagram is injective. According to point 2. in definition \ref{defmultstrict}, $\mathscr{M}/(y_1,\dots ,y_{\ell-1})\mathscr{M}$ is a strict graded $\mathbb{C}[z]$-module. Furthermore, the multiplication by $y_{\ell}$ on $\mathscr{M}/(y_1,\dots ,y_{\ell-1})\mathscr{M}$ is a morphism of degree zero graded $\mathbb{C}[z]$-modules. This implies that the multiplication by $y_{\ell}$ on $\mathscr{M}/(y_1,\dots ,y_{\ell-1})\mathscr{M}$ is injective. In the same way, $\mathscr{M}/(y_1,\dots ,y_{k})\mathscr{M}$ is a strict $\mathbb{C}[z]$-module, according to point 2. of definition \ref{defmultstrict}, the multiplication by $z$ is thus injective. The sequence $y_1,\dots ,y_k,z$ is therefore regular and this concludes the proof of the proposition.

\end{proof}

\begin{prop}

 Let $I\subset \{1,\dots ,p\}$ and let us denote $\mathbf{H}_I:=\{H_i\}_{i\in I}$. Let $\mathcal{M}$ be a coherent $\mathscr{D}_X$-module equipped with a coherent $F$-filtration such that $R_F\mathcal{M}$ underlies a mixed Hodge module. If $R_F\mathcal{M}$ is strictly $\mathbb{R}$-multispeciali\-sable along $\mathbf{H}$, then
\begin{itemize}
\item $R_F\mathcal{M}$ is strictly $\mathbb{R}$-multispecialisable along $\mathbf{H}_I$,
\item the canonical $V^{\mathbf{H}_I}$-multifiltration of $R_F\mathcal{M}$ satisfies: $V^{\mathbf{H}_I}_{\boldsymbol{\alpha}_I}R_F\mathcal{M}=\displaystyle
\sum_{\boldsymbol{\alpha}_{I^c} \in\mathbb{C}^{I^c}}V^{\mathbf{H}_I}_{\boldsymbol{\alpha}_I,\boldsymbol{\alpha}_{I^c}}R_F\mathcal{M}$,
\item for all $\boldsymbol{\alpha}_I\in \mathbb{R}^I$, $\textup{gr}^{\mathbf{H}^I}_{\boldsymbol{\alpha_I}}R_F\mathcal{M}$ is strictly $\mathbb{R}$-multispecialisable along $\mathbf{H}_{I^c}$.

\end{itemize}

\end{prop}

\begin{proof}

Following \cite[corollaire 1]{Maisonobe} $\mathcal{M}$ is $\mathbb{R}$-multispecialisable along $\mathbf{H}_I$. Moreover, following proposition \ref{strictmulcompat}, the filtrations $(F_\bullet\mathcal{M},V^1_\bullet\mathcal{M},\dots ,V^p_\bullet\mathcal{M})$ are compatible. Thus the filtrations $(F_\bullet\mathcal{M},\{V^i_\bullet\mathcal{M}\}_{i\in I})$ are compatible. Then, following proposition \ref{compatstrict}, we have that $R_F\mathcal{M}$ is strictly $\mathbb{R}$-multispecialisable along $\mathbf{H}_I$. 

We deduce from \eqref{multifilt} of the proof of proposition \ref{compatstrict} and from \cite[corollaire 1]{Maisonobe} that the canonical $V^{\mathbf{H}_I}$-multifiltration of $R_F\mathcal{M}$ satisfies: $V^{\mathbf{H}_I}_{\boldsymbol{\alpha}_I}R_F\mathcal{M}=\displaystyle
\sum_{\boldsymbol{\alpha}_{I^c} \in\mathbb{C}^{I^c}}V^{\mathbf{H}_I}_{\boldsymbol{\alpha}_I,\boldsymbol{\alpha}_{I^c}}R_F\mathcal{M}$.

The compatibility of the filtrations implies that $\textup{gr}^{\mathbf{H}^I}_{\boldsymbol{\alpha_I}}R_F\mathcal{M}=
R_F\textup{gr}^{\mathbf{H}^I}_{\boldsymbol{\alpha_I}}\mathcal{M}$ underlies a mixed Hodge module. Moreover, following \ref{appA}, the filtrations $(F_\bullet\textup{gr}^{\mathbf{H}^I}_{\boldsymbol{\alpha_I}}\mathcal{M},\{V^i_\bullet\textup{gr}^{\mathbf{H}^I}_{\boldsymbol{\alpha_I}}\mathcal{M}\}_{i\in I^c})$ are compatible. Following proposition \ref{compatstrict} we deduce that $\textup{gr}^{\mathbf{H}^I}_{\boldsymbol{\alpha_I}}R_F\mathcal{M}$ is strictly $\mathbb{R}$-multispecialisable along $\mathbf{H}_{I^c}$.

\end{proof}

\section{Direct image}

Here we show that strict $\mathbb{R}$-multispecialisability is preserved by a proper direct image. We also show that, under this condition, the canonical $V$-multifiltration and its graded pieces commute with proper direct image.

\begin{theorem}\label{commutVgrH}

Let $h:X\to Y$ be an holomorphic map between two complex manifolds, and $f=h\times \text{Id}:X\times \mathbb{C}^p\to Y\times \mathbb{C}^p$. Let $\mathbf{H}$ be the set of hyperplanes of equation $\{t_i=0\}$ for $1\leq i\leq p$ where the $t_i$ are the coordinates of $\mathbb{C}^p$. Let $\mathcal{M}$ be a coherent $\mathscr{D}_{X\times \mathbb{C}^p}$-module equipped with a good $F$-filtration. let us suppose that $R_F\mathcal{M}$ is strictly $\mathbb{R}$-multispecialisable along $\mathbf{H}$ and that $R_F\mathcal{M}$ underlies a mixed Hodge module.
  Then for all $k\in\mathbb{Z}$:
\begin{itemize}
\item the $R_F\mathscr{D}_{Y\times\mathbb{C}^p}$-module $\mathcal{H}^kf_*(R_F\mathcal{M})$ is strictly $\mathbb{R}$-multispecialisable along $\mathbf{H}$.
\item
 For all $\boldsymbol{\alpha}\in\mathbb{C}^p$, the canonical $V$-multifiltration for the functions $(t_1,\dots ,t_p)$ satisfies
\begin{equation}
\mathcal{H}^kf_*(V_{\boldsymbol{\alpha}}R_F\mathcal{M})\simeq V_{\boldsymbol{\alpha}}\mathcal{H}^kf_*(R_F\mathcal{M})
\label{commutVH}\end{equation}
and
\begin{equation}
\mathcal{H}^kf_*(\text{gr}^{V}_{\boldsymbol{\alpha}}R_F\mathcal{M})\simeq \text{gr}^{V}_{\boldsymbol{\alpha}}\mathcal{H}^kf_*(R_F\mathcal{M}).
\label{commutgrH}\end{equation}

\end{itemize}

\end{theorem}

\begin{proof}

We define a $V$-multifiltration of $\mathcal{H}^kf_*(R_F\mathcal{M})$ in the following way:
\begin{equation}U_{\boldsymbol{\alpha}}\mathcal{H}^kf_*(R_F\mathcal{M}):=\text{image}[\mathcal{H}^kf_*(V_{\boldsymbol{\alpha}}R_F\mathcal{M})\to \mathcal{H}^kf_*(R_F\mathcal{M})].
\label{deffiltrH}\end{equation}
First we will find locally, for all $1\leq i\leq p$, polynomials satisfying  
\[(U_{\boldsymbol{\alpha}}\mathcal{H}^kf_*(R_F\mathcal{M}))b_i(E_i-\alpha_iz)\subset (U_{\boldsymbol{\alpha-1}_i}\mathcal{H}^kf_*(R_F\mathcal{M}))\]
and with roots in $[-1,0[$. Second we will show that the natural morphism
\[\mathcal{H}^kf_*(V_{\boldsymbol{\alpha}}R_F\mathcal{M})\to \mathcal{H}^kf_*(R_F\mathcal{M})\]
is injective, that the $V$-multifiltration \eqref{deffiltrH} is a good $V$-multifiltration and that it satisfy the properties of the canonical $V$-multifiltration.
This will imply that $\mathcal{H}^kf_*(R_F\mathcal{M})$ is strictly $\mathbb{R}$-multispecialisable along $\mathbf{H}$ and this will provide us with the natural isomorphism \eqref{commutVH}. Finally, we will construct the isomorphism \eqref{commutgrH}.

\subsubsection*{Cancelling polynomials.}

Let $(y,t)\in Y\times \mathbb{C}^p$ and let $1\leq i\leq p$, we are going to find a polynomial $b_i(s)\in \mathbb{C}[s]$ satisfying for all $\boldsymbol{\alpha}\in\mathbb{C}^p$, 
$(U_{\boldsymbol{\alpha}}\mathcal{H}^kf_*(R_F\mathcal{M}))_{(y,t)}b_i(E_i-\alpha_iz)\subset (U_{\boldsymbol{\alpha-1}_i}\mathcal{H}^kf_*(R_F\mathcal{M}))_{(y,t)}$. Let $x\in f^{-1}(y)$, by the definition of the canonical $V$-multifiltration there exists an open neighbourhood $U$ of $(x,t)$ and a polynomial $b_i^x(s)\in \mathbb{C}[s]$ satisfying for all $\boldsymbol{\alpha}\in\mathbb{C}^p$, $\restriction{V_{\boldsymbol{\alpha}}R_F\mathcal{M}}{U}b_i^x(E_i-\alpha_iz)\subset \restriction{V_{\boldsymbol{\alpha-1}_i}R_F\mathcal{M}}{U}$ and with roots in $[-1,0[$. The map $f$ being proper, $f^{-1}(y)$ is compact, thus there exists a finite set $J$, points $x_j\in f^{-1}(y)$ for all $j\in I$ and an open neighbourhood $V$ of $f^{-1}(y)\times \{t\}$ such that
\[\restriction{V_{\boldsymbol{\alpha}}R_F\mathcal{M}}{V}\prod_{j\in J}b_i^{x_j}(E_i-\alpha_iz)\subset \restriction{V_{\boldsymbol{\alpha-1}_i}R_F\mathcal{M}}{V}.\]
By functoriality, the action of $E_i$ on the $\mathscr{D}_X$-module $V_{\boldsymbol{\alpha}}R_F\mathcal{M}/V_{\boldsymbol{\alpha-1}_i}R_F\mathcal{M}$ induce an action on
\[\left(U_{\boldsymbol{\alpha}}\mathcal{H}^kf_*(R_F\mathcal{M})\right)/\left(U_{\boldsymbol{\alpha-1}_i}
\mathcal{H}^kf_*(R_F\mathcal{M})\right)\]
and
\[(U_{\boldsymbol{\alpha}}\mathcal{H}^kf_*(R_F\mathcal{M}))_{(y,t)}\prod_{j\in J}b_i^{x_j}(E_i-\alpha_iz)\subset (U_{\boldsymbol{\alpha-1}_i}\mathcal{H}^kf_*(R_F\mathcal{M}))_{(y,t)}.\]
The roots of the polynomials $b_i^{x_j}(s)$ are in $[-1,0[$.

\subsubsection*{Injectivity.}

let us show that the natural morphism
\[\mathcal{H}^kf_*(V_{\boldsymbol{\alpha}}R_F\mathcal{M})\to \mathcal{H}^kf_*(R_F\mathcal{M})\]
is injective. We will show this by induction on the integer $p$, using property 7.5.8 of \cite{kaiserslautern}. We are going to show that the two following morphisms are injective:
\[\mathcal{H}^kf_*(V_{\boldsymbol{\alpha}}R_F\mathcal{M})
\underset{(1)}{\to} \mathcal{H}^kf_*(V^{\boldsymbol{H}_{2,\dots ,p}}_{\alpha_2,\dots ,\alpha_p}R_F\mathcal{M}) 
\underset{(2)}{\to} \mathcal{H}^kf_*(R_F\mathcal{M}),\]
where $V^{\boldsymbol{H}_{2,\dots ,p}}_{\alpha_2,\dots ,\alpha_p}R_F\mathcal{M}$ is the canonical $V$-multifiltration associated to the $p-1$ functions $t_2,\dots ,t_p$. let us denote $\mathcal{N}^\bullet:=f_*(V^{\boldsymbol{H}_{2,\dots ,p}}_{\alpha_2,\dots ,\alpha_p}R_F\mathcal{M})$ equipped with the $V$-filtration  $U_{\alpha_1}\mathcal{N}^\bullet:=f_*(V_{\boldsymbol{\alpha}}R_F\mathcal{M})$ for the hyperplane $\{t_1=0\}$. In order to consider $U_{\alpha_1}\mathcal{N}^\bullet$ as a sub-object of $\mathcal{N}^\bullet$, we use a definition of the functor $f_*$ involving explicit resolutions given in \cite[definition 4.3.3]{kaiserslautern}, namely $f_*R_F\mathcal{M}$ is the simple complex associated to the double complex
\[f_*R_F\mathcal{M}:=f_*\textup{God}^\bullet\left(R_F\mathcal{M}
\otimes_{R_F\mathscr{D}_{X\times\mathbb{C}^p}}
\textup{Sp}^\bullet_{X\times\mathbb{C}^p\to Y\times\mathbb{C}^p}(R_F\mathscr{D}_{X\times\mathbb{C}^p})\right).\]
Morphism (1) can be rewritten
\[\mathcal{H}^k(U_{\alpha_1}\mathcal{N}^\bullet)
\underset{(1)}{\to} \mathcal{H}^k(\mathcal{N}^\bullet).\]
let us check the three hypothesis of property 7.5.8 of \cite{kaiserslautern}:
\begin{enumerate}
\item let us show that $\text{gr}^{U}\mathcal{N}^\bullet$ is strict, namely for all $k\in\mathbb{N}$, $\mathcal{H}^k\text{gr}^{U}\mathcal{N}^\bullet$ is strict. We have, thanks to the compatibility of the filtrations, $\mathcal{H}^k\text{gr}^{U}\mathcal{N}^\bullet=\mathcal{H}^kf_*(V^{\boldsymbol{H}_{2,\dots ,p}}_{\alpha_2,\dots ,\alpha_p}\text{gr}^{H_1}_{\alpha_1}R_F\mathcal{M})$. By definition of mixed Hodge modules, $\text{gr}^{H_1}_{\alpha_1}R_F\mathcal{M}$ also satisfy the hypothesis of the theorem. By induction we thus get
\[\mathcal{H}^kf_*(V^{\boldsymbol{H}_{2,\dots ,p}}_{\alpha_2,\dots ,\alpha_p}\text{gr}^{H_1}_{\alpha_1}R_F\mathcal{M})\simeq V^{\boldsymbol{H}_{2,\dots ,p}}_{\alpha_2,\dots ,\alpha_p}\mathcal{H}^kf_*(\text{gr}^{H_1}_{\alpha_1}R_F\mathcal{M}).
\] 
Following the direct image theorem \cite[Proposition 2.16]{SaitoMHM}, $\text{gr}^{H_1}_{\alpha_1}R_F\mathcal{M}$ underlying a mixed Hodge module, $\mathcal{H}^kf_*(\text{gr}^{H_1}_{\alpha_1}R_F\mathcal{M})$ is strict. Thus $\text{gr}^{U}\mathcal{N}^\bullet$ is strict.
\item Strict $\mathbb{R}$-specialisability and $f$ being proper enable us to show, as above, the existence of cancelling polynomials of $\text{gr}^{U}\mathcal{N}^\bullet$ 
and thus the hypothesis of monodromy of property 7.5.8 of \cite{kaiserslautern}.
\item By the definition of the canonical $V$-multifiltration, for $\alpha_1<0$, right multiplication by $t_1$ gives an isomorphism $t_1:U_{\alpha_1}\mathcal{N}^j\xrightarrow{\sim} U_{\alpha_1-1}\mathcal{N}^j$ for all $j$.
\item Remark 4.3.4(4) of \cite{kaiserslautern} ensure that for $k>2\text{dim}X+p$ and for all $\alpha_1$, $\mathcal{H}^k(U_{{\alpha}_1}\mathcal{N}^\bullet)=0$.
\end{enumerate}
Thus, morphism (1) is injective. We apply the induction hypothesis to show that morphism (2) is injective. The natural morphism
\begin{equation}\mathcal{H}^kf_*(V_{\boldsymbol{\alpha}}R_F\mathcal{M})\to \mathcal{H}^kf_*(R_F\mathcal{M})
\label{commutVbisH}\end{equation}
is therefore injective. Thus we have

\begin{equation}\label{commutVU}
\mathcal{H}^kf_*(V_{\boldsymbol{\alpha}}R_F\mathcal{M})=
U_{\boldsymbol{\alpha}}\mathcal{H}^kf_*(R_F\mathcal{M})
\end{equation}

Furthermore property 7.5.8 of \cite{kaiserslautern} also provide the equality

\begin{equation}\label{commutVUgr}
\text{gr}_{\alpha_1}^{U_1}\mathcal{H}^kf_*(V_{\alpha_2,\dots ,\alpha_p}R_F\mathcal{M})=
\mathcal{H}^kf_*(V_{\alpha_2,\dots ,\alpha_p}\text{gr}_{\alpha_1}^{V_1}R_F\mathcal{M}).
\end{equation}

\subsubsection*{Good $V$-multifiltration property.}

let us show that the $V$-multifiltration \eqref{deffiltrH} is a good $V$-multifiltration.
We have to show that, for all $\boldsymbol{\alpha}\in\mathbb{C}^p$, the $V$-multifiltration  $U_{\boldsymbol{\alpha}+\bullet}(\mathcal{H}^kf_*(R_F\mathcal{M}))$ (indexed by $\mathbb{Z}^p$) is good. In order to lighten notations, we deal with the case $\boldsymbol{\alpha}=0$. Following \eqref{commutVU} we have
\begin{equation}
\mathcal{H}^kf_*(V_{{\boldsymbol{k}}}R_F\mathcal{M})\xrightarrow{\sim} 
U_{{\boldsymbol{k}}}(\mathcal{H}^kf_*(R_F\mathcal{M})).
\end{equation}
The compatibility property ensure that $V_{{\boldsymbol{k}}}R_F\mathcal{M}=R_F(V_{{\boldsymbol{k}}}\mathcal{M})$, thus $V_{{\boldsymbol{k}}}R_F\mathcal{M}$ is a coherent strict $V_{\boldsymbol{0}}\tilde{\mathscr{D}}_{Y\times\mathbb{C}^p}$-module. It is therefore filtered by the $\mathcal{O}_{Y\times\mathbb{C}^p}$-modules
\[\sum_{0\leq \ell\leq p} F_\ell V_{{\boldsymbol{k}}}\mathcal{M}z^\ell
\]
that give a good filtration. Since $f$ is proper we can then apply Grauert coherence theorem to these $\mathcal{O}_{Y\times\mathbb{C}^p}$-modules. We deduce that $\mathcal{H}^kf_*(V_{{\boldsymbol{k}}}R_F\mathcal{M})$ is $V_{\boldsymbol{0}}\tilde{\mathscr{D}}_{Y\times\mathbb{C}^p}$-coherent. Thus, $U_{{\boldsymbol{k}}}(\mathcal{H}^kf_*(R_F\mathcal{M}))$ is $V_{\boldsymbol{0}}\tilde{\mathscr{D}}_{Y\times\mathbb{C}^p}$-coherent.

We remark that, for all $1\leq i\leq p$ and ${\boldsymbol{k}}\in\mathbb{Z}^p$ satisfy $k_i\leq -1$, the isomorphism of multiplication by $t_i$
\[t_i:V_{\boldsymbol{k}}\mathcal{M}\xrightarrow{\sim} V_{\boldsymbol{k-1}_i}\mathcal{M}\]
induce an isomorphism 
\begin{equation}t_i:U_{{\boldsymbol{k}}}(\mathcal{H}^kf_*(R_F\mathcal{M}))
\xrightarrow{\sim} 
U_{\boldsymbol{k-1}_i}(\mathcal{H}^kf_*(R_F\mathcal{M})).\label{isotH}\end{equation}

In order to lighten notations we will denote here $\mathfrak{M}$ instead of $R_F\mathcal{M}$.
let us consider the short exact sequence
\[0 \to V_{\boldsymbol{k-1}_i}\mathfrak{M} \to V_{{\boldsymbol{k}}}\mathfrak{M} \to
V_{{\boldsymbol{k}}}\mathfrak{M}/V_{\boldsymbol{k-1}_i}\mathfrak{M}\to 0\]
and let us apply the functor $f_*$. We get the associated long exact sequence,
\begin{multline}\notag\dots \to\mathcal{H}^kf_*(V_{\boldsymbol{k-1}_i}\mathfrak{M}) \to 
\mathcal{H}^kf_*(V_{{\boldsymbol{k}}}\mathfrak{M}) \to
\mathcal{H}^kf_*(V_{{\boldsymbol{k}}}\mathfrak{M}/V_{\boldsymbol{k-1}_i}\mathfrak{M})\to \\
\mathcal{H}^{k+1}f_*(V_{\boldsymbol{k-1}_i}\mathfrak{M}) \to \mathcal{H}^{k+1}f_*(V_{{\boldsymbol{k}}}\mathfrak{M})\to\dots  \end{multline}
Following \eqref{commutVbisH}, the first and last morphisms are injective, we thus get the short exact sequence
\[0 \to U_{\boldsymbol{k-1}_i}\mathcal{H}^kf_*(\mathfrak{M}) \to U_{\boldsymbol{k}}\mathcal{H}^kf_*(\mathfrak{M}) \to
\mathcal{H}^kf_*(V_{{\boldsymbol{k}}}\mathfrak{M}
/V_{\boldsymbol{k-1}_i}\mathfrak{M})\to 0.\]

let us consider the following commutative diagram:

\[\xymatrix{0 \ar[r] &
 U_{\boldsymbol{k-1}_i}\mathcal{H}^kf_*(\mathfrak{M}) \ar[r] \ar[d]^{\cdot\eth_{t_i}} &  U_{{\boldsymbol{k}}}\mathcal{H}^kf_*(\mathfrak{M}) \ar[r] \ar[d]^{\cdot\eth_{t_i}} & 
\mathcal{H}^kf_*(V_{{\boldsymbol{k}}}\mathfrak{M}/V_{\boldsymbol{k-1}_i}\mathfrak{M})\ar[r] \ar[d]^{\cdot\eth_{t_i}} &  0 \\
0 \ar[r] &
 U_{{\boldsymbol{k}}}\mathcal{H}^kf_*(\mathfrak{M}) \ar[r] &  U_{\boldsymbol{k+1}_i}\mathcal{H}^kf_*(\mathfrak{M}) \ar[r] & 
\mathcal{H}^kf_*(V_{\boldsymbol{k+1}_i}\mathfrak{M}/V_{{\boldsymbol{k}}}\mathfrak{M})\ar[r] &  0.
}\]
For $k_i\geq 0$, following the properties of the canonical $V$-multifiltration, the third vertical arrow of the diagram is an isomorphism. Using the snake lemma, we deduce that 
\[U_{\boldsymbol{k+1}_i}\mathcal{H}^kf_*(\mathfrak{M})=
U_{{\boldsymbol{k}}}\mathcal{H}^kf_*(\mathfrak{M})
+U_{{\boldsymbol{k}}}\mathcal{H}^kf_*(\mathfrak{M})\cdot\eth_{t_i}.\]
This result and the isomorphism \eqref{isotH} give the good $V$-multifiltration property of \[U_{{\bullet}}\mathcal{H}^kf_*(R_F\mathcal{M}).\]

\subsubsection*{Graded pieces.}

We are going to show that for all $I\subset\{1,\dots ,p\}$ and all $\boldsymbol{\alpha}\in\mathbb{R}^p$,
\begin{equation}\label{commutUgrV}
\frac{U_{\boldsymbol{\alpha}}\mathcal{H}^kf_*(R_F\mathcal{M})}{U_{<\boldsymbol{\alpha}_{I},\boldsymbol{\alpha}_{I^c}}\mathcal{H}^kf_*(R_F\mathcal{M})}\simeq U_{\boldsymbol{\alpha}_{I^c}}\mathcal{H}^kf_*
(\textup{gr}_{\boldsymbol{\alpha}_I}^{\mathbf{H}_I}R_F\mathcal{M}).
\end{equation} 

We use an induction on the cardinality if $I$. Upon renumbering we can suppose that $1\in I$. If we omit to write $\mathcal{H}^kf_*(R_F\mathcal{M})$ we have

\begin{equation}\label{divis}
\frac{U_{\boldsymbol{\alpha}}}{U_{<\boldsymbol{\alpha}_{I},\boldsymbol{\alpha}_{I^c}}}=
\left.
\frac{U_{\boldsymbol{\alpha}}}{U_{<\alpha_1,\boldsymbol{\alpha}_{\{1\}^c}}}
\middle/\sum_{i\in I-\{1\}}
\frac{U_{\alpha_1,\dots ,<\alpha_i,\dots }}{U_{<\alpha_1,\dots ,<\alpha_i,\dots }}
\right.
\end{equation}
Following \eqref{commutVU}, \eqref{commutVUgr} and the compatibility of the canonical filtrations of $R_F\mathcal{M}$  we have 

\[
\frac{U_{\boldsymbol{\alpha}}}{U_{<\alpha_1,\boldsymbol{\alpha}_{\{1\}^c}}}=
\text{gr}_{\alpha_1}^{U_1}\mathcal{H}^kf_*(V_{\alpha_2,\dots ,\alpha_p}R_F\mathcal{M})=
\mathcal{H}^kf_*(V_{\alpha_2,\dots ,\alpha_p}\text{gr}_{\alpha_1}^{V_1}R_F\mathcal{M})=
U_{\alpha_2,\dots ,\alpha_p}\mathcal{H}^kf_*(\text{gr}_{\alpha_1}^{V_1}R_F\mathcal{M}).
\]
In the same way we get, using \eqref{divis},

\[\frac{U_{\boldsymbol{\alpha}}}{U_{<\boldsymbol{\alpha}_{I},\boldsymbol{\alpha}_{I^c}}}=
\frac{U_{\boldsymbol{\alpha}_{ \{1\}^c}}\mathcal{H}^kf_*(\text{gr}_{\alpha_1}^{V_1}R_F\mathcal{M})}
{U_{<\boldsymbol{\alpha}_{I-\{1\}},\boldsymbol{\alpha}_{I^c}}\mathcal{H}^kf_*(\text{gr}_{\alpha_1}^{V_1}R_F\mathcal{M})}.
\]
Since $\text{gr}_{\alpha_1}^{V_1}R_F\mathcal{M}$ stisfy the hypothesis of the theorem we have, using the induction hypothesis,

\[\frac{U_{\boldsymbol{\alpha}}\mathcal{H}^kf_*(R_F\mathcal{M})}{U_{<\boldsymbol{\alpha}_{I},\boldsymbol{\alpha}_{I^c}}\mathcal{H}^kf_*(R_F\mathcal{M})}=
U_{\boldsymbol{\alpha}_{I^c}}\mathcal{H}^kf_*
(\textup{gr}_{\boldsymbol{\alpha}_{I-\{1\}}}^{\mathbf{H}_{I-\{1\}}}\text{gr}_{\alpha_1}^{V_1}R_F\mathcal{M})=
U_{\boldsymbol{\alpha}_{I^c}}\mathcal{H}^kf_*
(\textup{gr}_{\boldsymbol{\alpha}_I}^{\mathbf{H}_I}R_F\mathcal{M})
\]
where the last equality follow from the compatibility of the filtrations.

\subsubsection*{Canonical $V$-multifiltration properties.}

let us show that the $V$-multifiltration $U_{{\bullet}}\mathscr{H}^kf_*(R_F\mathcal{M})$
satisfies the properties of definition \ref{defmultstrict} of the canonical $V$-multifiltration. We have already shown the propertiy of good $V$-multifiltration.

Following the isomorphism \eqref{commutUgrV}, the module 
\[\frac{U_{\boldsymbol{\alpha}}\mathscr{H}^kf_*(R_F\mathcal{M})}{U_{<\boldsymbol{\alpha}_{I},\boldsymbol{\alpha}_{I^c}}\mathscr{H}^kf_*(R_F\mathcal{M})}\]
is a sub-object of $\mathscr{H}^kf_*
(\textup{gr}_{\boldsymbol{\alpha}_I}^{\mathbf{H}_I}R_F\mathcal{M})$. Yet, since $R_F\mathcal{M}$ underlies a mixed Hodge module, $\textup{gr}_{\boldsymbol{\alpha}_I}^{\mathbf{H}_I}R_F\mathcal{M}$ also do. Thus its direct image is strict and we deduce that
\[\frac{U_{\boldsymbol{\alpha}}\mathscr{H}^kf_*(R_F\mathcal{M})}{U_{<\boldsymbol{\alpha}_{I},\boldsymbol{\alpha}_{I^c}}\mathscr{H}^kf_*(R_F\mathcal{M})}\]
is strict for all $I\subset\{1,\dots ,p\}$ and all $\boldsymbol{\alpha}\in \mathbb{R}^p$.

We deduce surjectivity properties of the multiplication by $t_i$ and $\eth_i$ from the isomorphism
\[\mathscr{H}^kf_*(\text{gr}_{\boldsymbol{\alpha}}R_F\mathcal{M})\simeq \text{gr}^{U}_{\boldsymbol{\alpha}}\mathscr{H}^kf_*(R_F\mathcal{M}),\]
and from the following facts
\begin{itemize}
\item for all $1\leq i\leq p$ and all $\alpha_i<0$, the morphism $t_i:\textup{gr}_{\boldsymbol{\alpha}}R_F\mathcal{M}\to \textup{gr}_{\boldsymbol{\alpha-1}_i}R_F\mathcal{M}$ is an isomorphism,
\item for all $1\leq i\leq p$ and all $\alpha_i>-1$, the morphism $\eth_i:\textup{gr}_{\boldsymbol{\alpha}}R_F\mathcal{M}\to \textup{gr}_{\boldsymbol{\alpha+1}_i}R_F\mathcal{M}(-1)$ is an isomorphism.
\end{itemize}

Thus, the $V$-multifiltration 
\[U_{{\bullet}}\mathscr{H}^kf_*(R_F\mathcal{M})\]
is the canonical $V$-multifiltration, $\mathscr{H}^kf_*(R_F\mathcal{M})$ is therefore strictly $\mathbb{R}$-multispecialisable along $\mathbf{H}$ and we have the following isomorphisms
\[
\mathscr{H}^kf_*(V_{\boldsymbol{\alpha}}R_F\mathcal{M})\simeq V_{\boldsymbol{\alpha}}\mathscr{H}^kf_*(R_F\mathcal{M})
\]
and
\[
\mathscr{H}^kf_*(\text{gr}^V_{\boldsymbol{\alpha}}R_F\mathcal{M})\simeq \text{gr}^V_{\boldsymbol{\alpha}}\mathscr{H}^kf_*(R_F\mathcal{M}).
\]

\end{proof}

\section{The direct image of relative monodromy filtrations}

Let $X$ be a complex manifold, $\mathscr{M}\in \textup{MH}(X\times \mathbb{C}^p_{\boldsymbol{t}})$ a pure Hodge module on $X\times \mathbb{C}^p_{\boldsymbol{t}}$ and $S(.,.)$ a polarisation of $\mathscr{M}$. We suppose that $\mathscr{M}$ is strictly $\mathbb{R}$-multispecialisable along $\mathbf{H}$ where $H_i:=\{t_i=0\}$. The mixed Hodge module $\Psi_{\boldsymbol{t}}\mathscr{M}$ is equipped with $p$ nilpotent endomorphisms $N_i$, for $1\leq i\leq p$.

\begin{defi}

Let $A$ be an object in and abelian category equipped with $p$ nilpotent endomorphisms  $N_1,\dots ,N_p$. We say that $(A,N_1,\dots ,N_p)$ satisfy property $(MF)$ if the relative monodromic filtrations (definition \ref{defifiltmonorelat}) satisfy
\[W(N_1,W(N_2,W(\dots W(N_p))))A=W(N_1+\dots +N_p)A.\] 

\end{defi}

\begin{rema}

Following \cite[4.12]{NilpSchmid} and \cite[Proposition 4.72]{CKS} this condition is satified in the case of a nilpotent orbit.

\end{rema}

\begin{theo}\label{imagedirecteW}

Let $f:X\times \mathbb{C}^p_{\boldsymbol{t}}\to Y\times \mathbb{C}^p_{\boldsymbol{t}}$ be a projective morphism and let us supose that $(\Psi_{\boldsymbol{t}}\mathscr{M},N_1,\dots ,N_p)$ satisfy property $(MF)$, then $(\Psi_{\boldsymbol{t}}\mathscr{H}^jf_*\mathscr{M},N_1,\dots ,N_p)$ also satisfy $(MF)$.

\end{theo}

First we will prove a lemma.

\begin{lemm}\label{pHLP}

Let $(\mathscr{M},N_1,\dots ,N_p,S)$ be a polarized $p$-graded Hodge-Lefschetz module with polarization $S$ (definition \ref{defiMHLP}). If we define the $(p-1)$-graduation 
\[\mathscr{M}_{\ell,\ell_3,\dots ,\ell_p}=\bigoplus_{\ell_1+\ell_2=\ell} \mathscr{M}_{\ell_1,\ell_2,\ell_3,\dots ,\ell_p},\]
then $(\mathscr{M},N_1+N_2,N_3,\dots ,N_p,S)$ is a polarized graded Hodge-Lefschetz module with polarization $S$.

\end{lemm}

\begin{proof}

We will make an induction on the dimension of the support of $\mathscr{M}$. Using Kshiwara equivalence, the $0$ dimension case correspond to the case of polarised Hodge-Lefschetz structures of proposition \ref{SHLP}. 

Let $g$ be a germ of holomorphic function. The definition by induction \ref{defiMHL} of polarised graded H-L modules implies that
\[(P_\ell\Psi_{g,\lambda}\mathscr{M},P_\ell\Psi_{g,\lambda}N_1,\dots ,P_\ell\Psi_{g,\lambda}N_p,N,P_\ell\Psi_{g,\lambda}S)
\]
is a $(p+1)$-graded polarised H-L module. Locally, we have a decomposition
\[\mathscr{M}=\bigoplus_{Z}\mathscr{M}_Z
\]
where $Z$ is irreducible and $\mathscr{M}_Z$ has strict support $Z$. If $g$ is zero on $Z$, then $\Psi_{g,\lambda}\mathscr{M}_Z=0$ for all $\lambda$, otherwise $\Psi_{g,\lambda}\mathscr{M}_Z$ has codimension $1$ support in $Z$. Thus dimension of the support has strictly decrease and, by induction, we get that 
\[\begin{array}{cll}
& & (P_\ell\Psi_{g,\lambda}\mathscr{M},(P_\ell\Psi_{g,\lambda}N_1+P_\ell\Psi_{g,\lambda}N_2),
P_\ell\Psi_{g,\lambda}N_3,\dots ,P_\ell\Psi_{g,\lambda}N_p,N,P_\ell\Psi_{g,\lambda}S)\\
& = & (P_\ell\Psi_{g,\lambda}\mathscr{M},P_\ell\Psi_{g,\lambda}(N_1+N_2),P_\ell\Psi_{g,\lambda}N_3,\dots ,P_\ell\Psi_{g,\lambda}N_p,N,P_\ell\Psi_{g,\lambda}S)
\end{array}
\]
is a bigraded polarised H-L module. Thus, $(\mathscr{M},N_1+N_2,N_3,\dots ,N_p,S)$ satisfies the inductive definition of graded polarised Hodge-Lefschetz modules.

\end{proof}

\begin{proof}[Proof of theorem \ref{imagedirecteW}]

We have to prove that
\begin{equation}\label{pfiltW}
W(N_1,W(N_2,W(\dots W(N_p))))\Psi_{\boldsymbol{t}}\mathscr{H}^jf_*\mathscr{M}
=W(N_1+\dots +N_p)\Psi_{\boldsymbol{t}}\mathscr{H}^jf_*\mathscr{M}.
\end{equation}

let us denote $\mathscr{N}:=\Psi_{\boldsymbol{t}}\mathscr{M}$ and $W\mathscr{N}:=W(N_1,W(N_2,W(\dots W(N_p))))\mathscr{N}=W(N_1+\dots +N_p)\mathscr{N}$. Following the strict $\mathbb{R}$-multispecialisability and theorem \ref{commutVgrH}, we have
\[\Psi_{\boldsymbol{t}}\mathscr{H}^jf_*\mathscr{M}\simeq \mathscr{H}^jf_*\Psi_{\boldsymbol{t}}\mathscr{M}=\mathscr{H}^jf_*\mathscr{N}.
\]
We define the following filtration:
 
\begin{equation}\label{pdefW}
W\Psi_{\boldsymbol{t}}\mathscr{H}^jf_*\mathscr{M}:=
\text{Im}(\mathscr{H}^jf_*W\mathscr{N}\to \mathscr{H}^jf_*\mathscr{N}).
\end{equation}

We are going to show that this filtration is the same as the two filtrations of \eqref{pfiltW}. Following \cite[Theorem 5.3.1 and Proposition 5.3.4]{HM1}, the filtration 
\[\text{Im}(\mathscr{H}^jf_*W(N_p)\Psi_{t_p}\mathscr{M}\to \mathscr{H}^jf_*\Psi_{t_p}\mathscr{M})
\]
of $\mathscr{H}^jf_*\Psi_{t_p}\mathscr{M}$ is the filtration associated to $N_p$. Furthermore, by definition of pure Hodge modules, we can consider the mixed Hodge module $(\Psi_{t_p}\mathscr{M},W(N_p))$. Proposition 2.16 of \cite{SaitoMHM} for mixed Hodge modules ensure that the filtration
\[\text{Im}(\mathscr{H}^jf_*W(N_{p-1},W(N_p))\Psi_{t_{p-1}}
\Psi_{t_p}\mathscr{M}\to \mathscr{H}^jf_*\Psi_{t_{p-1}}\Psi_{t_p}\mathscr{M})
\]
is equal to
\[W(N_{p-1},W(N_{p}))\Psi_{t_{p-1}}\mathscr{H}^jf_*
\Psi_{t_p}\mathscr{M}\simeq 
W(N_{p-1},W(N_{p}))\Psi_{t_{p-1}}\Psi_{t_p}\mathscr{H}^jf_*
\mathscr{M}.\]
Thus, using an induction and applying direct image theorems at each step, we can get the equality between the first filtration of \eqref{pfiltW} and the filtration \eqref{pdefW}:
\begin{equation}\label{pegal1}
W\Psi_{\boldsymbol{t}}\mathscr{H}^jf_*\mathscr{M}\simeq 
W(N_1,W(N_2,W(\dots W(N_p))))\Psi_{\boldsymbol{t}}
\mathscr{H}^jf_*\mathscr{M}.
\end{equation}

For the second filtration of \eqref{pfiltW}, we are going to use lemma \ref{pHLP}.
let us consider the following polarized bigraded H-L module:
\[\left(\bigoplus_{\boldsymbol{\ell}} \textup{Gr}_{\ell_1}^{W(N_1)}\dots \textup{Gr}_{\ell_p}^{W(N_p)}\mathscr{N},
N_1,\dots ,N_p,
S\right).
\]
Iterating \cite[Theorem 3.2.9]{Kash86}, we get
\begin{equation}
\textup{Gr}_{\ell}^{W(N_1,W(N_2,W(\dots W(N_p))))}\mathscr{N}  = 
\bigoplus_{k_1+\dots +k_p=\ell} \textup{Gr}_{k_1}^{W(N_1)}\dots \textup{Gr}_{k_p}^{W(N_p)}\mathscr{N}.
\end{equation}
The equality $W(N_1,W(N_2,W(\dots W(N_p))))\mathscr{N}=W(N_1+\dots +N_p)\mathscr{N}$ and lemma \ref{pHLP} imply that
\[
\left(\bigoplus_{\ell} \textup{Gr}_{\ell}^{W(N_1+\dots +N_p)}\mathscr{N},
N_1+\dots +N_p,
S\right)
\]
is a polarized graded H-L module. The direct image theorem \cite[Proposition 5.3.5]{HM1} ensure then that the filtration 
\[\text{Im}(\mathscr{H}^jf_*W(N_1+\dots +N_p)\mathscr{N}\to \mathscr{H}^jf_*\mathscr{N})\]
is equal to the monodromy filtration
\[W(N_1+\dots +N_p)\mathscr{H}^jf_*\mathscr{N}.\]
Thus we have the equality
\begin{equation}\label{pegal2}
W\Psi_{\boldsymbol{t}}\mathscr{H}^jf_*\mathscr{M}\simeq 
W(N_1+\dots +N_p)\Psi_{\boldsymbol{t}}\mathscr{H}^jf_*\mathscr{M}.
\end{equation}
Combining \eqref{pegal1} and \eqref{pegal2}, we get the expected equality
\[W(N_1,W(N_2,W(\dots W(N_p))))\Psi_{\boldsymbol{t}}
\mathscr{H}^jf_*\mathscr{M}
=W(N_1+\dots +N_p)\Psi_{\boldsymbol{t}}\mathscr{H}^jf_*\mathscr{M}.\]

\end{proof}

\section{Quasi-ordinary hypersurface singularities}

We apply here the above results to some mixed Hodge modules with support a quasi-ordinary hypersurface singularity.

\begin{defi}

A germ of a complex analytic space $(Z,0)$ of dimension $p$ is a \emph{quasi-ordinary} singularity if there exists a finite projection $(Z,0)\to (\mathbb{C}^{p},0)$ unramified outside a normal crossing divisor in $(\mathbb{C}^{p},0)$.

If $(Z,0)$ is a irreducible quasi-ordinary hypersurface singularity, then there exists a parametrisation
\[\begin{array}{ccccc}
y & :& \mathbb{C}^p_{\boldsymbol{t}} & \to & \mathbb{C}^{p+1}_{\boldsymbol{z}}\\
& & (t_1,\dots ,t_p) & \mapsto & (t_1^k,\dots ,t_p^k,H(t_1,\dots ,t_p))
\end{array}\]
where $H$ is a convergent series.

\end{defi}

Let $D:=\{t_1\dots t_p=0\}\subset \mathbb{C}^p_{\boldsymbol{t}}$ and $U:=\mathbb{C}^p_{\boldsymbol{t}} \setminus D$. 

\begin{prop}

Let $(\mathcal{M},F_\bullet\mathcal{M})$ be a filtered $\mathscr{D}_X$-module underlying a polarizable pure Hodge module $\mathscr{M}$ with strict support $\mathbb{C}^p_{\boldsymbol{t}}$ and with smooth restriction to $U$ (in other words $\mathscr{M}$ is the intermediate extension of a polarised variation of Hodge structures with support $U$). Then the direct image $y_+R_F\mathcal{M}$, wich also underlies a polarizable pure Hodge module, is strictly $\mathbb{R}$-multispecialisable along the family $\mathbf{H}_{\boldsymbol{z}}$ where $H_{z_i}:=\{z_i=0\}$ pour $1\leq i\leq p$. Furthermore, for all $j\in\mathbb{Z}$, the Hodge-Lefschetz module $(\textup{gr}_{\boldsymbol{\alpha}}^{V}\mathscr{H}^jy_+\mathscr{M},N_1,\dots ,N_p)$, satisfies property $(FM)$.

\end{prop}

\begin{proof}

let us consider the family of hyperplanes $\mathbf{H}_{\boldsymbol{t}}$ given by the coordinate hyperplanes $H_{t_i}:=\{t_i=0\}$ for $1\leq i\leq p$.
Following \cite[3.b]{SaitoMHM}, the filtrations of $\mathcal{M}$ $(F_\bullet\mathcal{M},V^{H_{t_1}}_\bullet\mathcal{M},\dots ,V^{H_{t_i}}_\bullet\mathcal{M})$ are compatible. Moreover, using definition \cite[Théorème 3]{Maisonobe},
we get that $\mathcal{M}$ is $\mathbb{R}$-multispecialisable along $\mathbf{H}_{\boldsymbol{t}}$. Following proposition \ref{compatstrict}, $R_F\mathcal{M}$ is strictly $\mathbb{R}$-multispecialisable along $\mathbf{H}_{\boldsymbol{t}}$.

Let
\[\begin{array}{ccccc}
i & :& \mathbb{C}^p_{\boldsymbol{t}}  & \to & \mathbb{C}^p_{\boldsymbol{t}}\times\mathbb{C}^{p}_{\boldsymbol{z}}\\
& & (t_1,\dots ,t_p) & \mapsto & (t_1,\dots ,t_p,t_1^k,\dots ,t_p^k),
\end{array}\]
be the inclusion by the graph map. Following \cite[Théorème 3.4]{SaitoMHM}, $i_+R_F\mathcal{M}$ is also strictly $\mathbb{R}$-multispecialisable along $\mathbf{H}_{\boldsymbol{z}}$.

Following theorem \ref{commutVgrH}, $(y\times id_{\mathbb{C}_{\boldsymbol{z}}^p})_+i_+R_F\mathcal{M}$ is strictly $\mathbb{R}$-multispecialisable along $\mathbf{H}_{\boldsymbol{z}}$. 
Yet, $(y\times id_{\mathbb{C}_z^p})\circ i=j\circ y$ where
\[\begin{array}{ccccc}
j & :& \mathbb{C}^{p+1}_{\boldsymbol{z}}  & \to & \mathbb{C}^{p+1}_{\boldsymbol{z}}\times\mathbb{C}^{p}_{\boldsymbol{z}}\\
& & (z_1,\dots ,z_{p+1}) & \mapsto & (z_1,\dots ,z_{p+1},z_1,\dots ,z_p)
\end{array}\]
is the inclusion of the graph of the projection on the $p$ first coordinates. Thus $j_+y_+R_F\mathcal{M}$ is strictly $\mathbb{R}$-multispecialisable along $\mathbf{H}_{\boldsymbol{z}}$. 

Following the nilpotent orbit theorem \cite[4.12]{NilpSchmid} and \cite[Proposition 4.72]{CKS},$(\textup{gr}_{\boldsymbol{\alpha}}^{V}\mathscr{M},N_1,\dots ,N_p)$ satisfy property $(FM)$. Theorem \ref{imagedirecteW} implies then that $(\textup{gr}_{\boldsymbol{\alpha}}^{V}\mathscr{H}^jy_+\mathscr{M},N_1,\dots ,N_p)$ satisfy property $(FM)$.

\end{proof}

\begin{coro}

With the notation of the above proposition, the filtrations
\[(F_\bullet y_+\mathcal{M},V^{H_{z_1}}_\bullet y_+\mathcal{M},\dots ,
V^{H_{z_i}}_\bullet y_+\mathcal{M})\]
are compatible.

\end{coro}

\begin{proof}

This is a consequence of theorem \ref{commutVgrH} and proposition \ref{strictmulcompat}.

\end{proof}

\appendix

\addtocontents{toc}{\protect\setcounter{tocdepth}{1}}
\part{Appendix}

\section{Some notions on mixed Hodge modules}

We recollect some definitions and properties for mixed Hodge modules that are used in this article.

\subsection{Definitions}

\subsubsection*{Inverse image}

We give two definitions of the inverse image for mixed Hodge modules depending on the type of the morphism involved.

\begin{defi}\label{defimageinverseH}

Let $\mathscr{M}\in \textup{MHM}(Y,\mathbb{Q})$ and let $i:X\hookrightarrow Y$ be a closed immersion such that $X=\cap g_i^{0}$ where $g_1,...,g_r$ are holomorphic functions on $Y$. We denote $U_i:=\{g_i\neq 0\}$ and, for $I\subset \{1,...,r\}$, 
\[j_I:U_I:=\bigcap_{i\in I}U_i\to Y
\]
the inclusion. We define $i_*\mathcal{H}^j i^* \mathscr{M}$ as the $j$-t cohomology of the complex
\[\bigoplus_{|I|=-\bullet}(j_I)_!j_I^{-1}\mathscr{M}.
\]

\end{defi}

\begin{defi}\label{imageinverse!}

Let $f:X\to Y$ be a smooth morphism between two complex manifolds with $\text{dim}X-\text{dim}Y=\ell$ and $\mathscr{M}=(\mathcal{M},F,\mathcal{F};W)\in \textup{MHW}(X,\mathbb{Q})$. We define $\mathscr{H}^{-\ell}f^{!}\mathscr{M}=(\mathcal{M}',F',\mathcal{F}';W')$ with
\[\begin{array}{ccc}
(\mathcal{M}',F';W')=f^{*}(\mathcal{M},F;W[-\ell]) \\
(\mathcal{F}';W')=f^{!}(\mathcal{F};W[-\ell])[-\ell].
\end{array}
\]
We also define 
\[\mathscr{H}^{\ell}f^{*}\mathscr{M}
=\mathscr{H}^{-\ell}f^{!}\mathscr{M}(-\ell)\]

\end{defi}

\subsubsection*{Localisation}

Let $j:Y\hookrightarrow X$ be the inclusion of a smooth hypersurface, we consider the Kashiwara-Malgrange filtration along this hypersurface. We define the localisation along $Y$

\begin{defi}\label{localisationF}

Met $\mathscr{M}=(\mathcal{M},F,\mathcal{F};W)\in \textup{MHM}^{(p)}(X,\mathbb{Q})$, we define the naive localisation of $R_F\mathcal{M}$ as
\[R_F\mathcal{M}(*Y):=R_F\mathcal{M}\otimes_{\tilde{\mathscr{O}}_X}\tilde{\mathscr{O}}_X(*Y).
\]
Let
\[V_0R_F\mathcal{M}(*Y):=V_{-1}(R_F\mathcal{M})t^{-1}\subset R_F\mathcal{M}(*Y)
\]
we define the localisation of $R_F\mathcal{M}$ as
\[R_F\mathcal{M}[*Y]:=V_0R_F\mathcal{M}(*Y)\cdot\tilde{\mathscr{D}}_X.
\]

\end{defi}

\begin{prop}\label{isoDH}

Let $\varphi_H:\mathscr{M}\to\mathscr{N}$ be a morphism in the category $\textup{MHM}(X,\mathbb{Q})$. Let's denote $\varphi_D:\mathcal{M}\to\mathcal{N}$ the underlying morphism in the category of holonomic $\mathscr{D}_X$-modules. If $\varphi_D$ is an isomorphism, then $\varphi_H$ is also an isomorphism.

\end{prop}

\begin{proof}

The category $\textup{MHM}(X,\mathbb{Q})$ being abelian, we can consider $\mathscr{K}$ and $\mathscr{C}$ the kernel and the cokernel of $\varphi_H$. The morphism $\varphi_D$ being an isomorphism, and the forgetful functor being exact, the underlying holonomic  $\mathscr{D}_X$-modules of $\mathscr{K}$ and $\mathscr{C}$ are zero. Furthermore, the forgetful functor is faithful, this implies that $\mathscr{K}$ and $\mathscr{C}$ are zero and thus that 
$\varphi_H$ is an isomorphism.

\end{proof}

\subsection{Polarized $p$-graded Hodge-Lefschetz structures}

\subsubsection{Relative monodromy filtration}\label{relatmonofilt}

\begin{prop}

Let $\mathscr{A}$ be an abelian category and $H\in \mathscr{A}$ equipped with a nilpotent endomorphism $N$. There exists a unique filtration $M_\bullet H$ satisfying: 
\begin{enumerate}
\item for all $\ell\in\mathbb{Z}$, $N(M_\ell H)\subset M_{\ell-2} H$
\item for all $\ell\geq 1$, $N^\ell$ induces an isomorphism, 
$\textup{gr}^M_\ell H \xrightarrow{\sim} \textup{gr}^M_{-\ell} H$.
\end{enumerate}
This filtration is called the \emph{monodromy filtration} associated to $N$.
\end{prop}

\begin{proof}

\cite[Lemma-Definition 11.9]{P-S08}

\end{proof}

\begin{prop}\label{defifiltmonorelat}

Let $\mathscr{A}$ be an abelian category and $(H,L_\bullet H)$ a filtered object in $\mathscr{A}$ equipped with a filtered nilpotent endomorphism $N$. There exists at the most one filtration $M_\bullet H$ satisfying:
\begin{enumerate}
\item for all $\ell\in\mathbb{Z}$, $N(M_\ell H)\subset M_{\ell-2} H$,
\item for all $\ell\geq 1$, and all $k\in\mathbb{Z}$, $N^\ell$ induces an isomorphism,
$\textup{gr}^M_{\ell+k} \textup{gr}^L_{k}H \xrightarrow{\sim} \textup{gr}^M_{-\ell+k} \textup{gr}^L_{k}H$.
\end{enumerate}
If this filtration exists, it is called the \emph{relative monodromy filtration} associated to $N$.

\end{prop}

\begin{proof}

\cite[1.1]{SaitoMHM}

\end{proof}

\subsubsection{Polarized graded Lefschetz structures}

\begin{defi}

A \emph{graded Lefschetz structure} consists of a graded $\mathbb{R}$-vector space $\mathcal{H}$ equipped with a nilpotent endomorphism $N$ satisfying, for all $\ell\in\mathbb{Z}$ 
\[\mathcal{H}_\ell=\textup{gr}_\ell^M\mathcal{H}\]
where $M_\bullet$ is the monodromy filtration associated to $N$. For $\ell\in\mathbb{N}$, we denote $P_\ell\mathcal{H}:=\ker N^{\ell+1}:\mathcal{H}_\ell \to \mathcal{H}_{-\ell}$ the primitive part of degree $\ell$.

\end{defi}

The Lie algebra $\mathfrak{sl}(2,\mathbb{R})$ has basis:
\[
X=
  \begin{pmatrix}
    0 & 0 \\
    1 & 0
  \end{pmatrix},
~~Y=
  \begin{pmatrix}
    0 & 1 \\
    0 & 0
  \end{pmatrix},
~~H=
  \begin{pmatrix}
    -1 & 0 \\
     0 & 1
  \end{pmatrix}.
\]
With the following formulas:
\[
[X,Y]=H,~~[X,H]=-2X,~~[Y,H]=2Y.
\]

\begin{lemm}\label{proprepr}

\begin{enumerate}

\item Let $(\mathcal{H},N)$ be a graded Lefschetz structure. There exists a unique representation $\rho:\mathfrak{sl}(2,\mathbb{R})\to \mathcal{H}$ such that $\rho(Y)=N$ and, for all $\ell\in\mathbb{Z}$, $\mathcal{H}_\ell$ is the eigenspace of $\rho(H)$ for the eigenvalue $\ell$.
\item Let $M\in\text{End}(\mathcal{H})$, if $M$ commutes with $\rho(H)$ and $\rho(Y)$, then $M$ commutes with $\rho(X)$.

\end{enumerate}

\end{lemm}

\begin{proof}

\begin{enumerate}

\item \cite[Lemma 1.21]{P-S08}
\item For all $t\in\mathbb{R}$, the endomorphisms $(e^{-tM}\rho(X)e^{tM},\rho(Y),\rho(H))$ are also a representation of $\mathfrak{sl}(2,\mathbb{R})$, by uniqueness we have $e^{-tM}\rho(X)e^{tM}=\rho(X)$. This being true for all $t\in\mathbb{R}$ we have $[\rho(X),M]=0$.

\end{enumerate}

\end{proof}

\begin{defi}

For a graded Lefschetz structure $(\mathcal{H},N)$ with representation of $\mathfrak{sl}(2,\mathbb{R})$ denoted $\rho$, we define $w:=e^{-\rho(X)}e^{\rho(Y)}e^{-\rho(X)}$.

\end{defi}

\begin{defi}

A \emph{polarization} of a graded Lefschetz structure is a graded bilinear form
\[k:\mathcal{H}\otimes\mathcal{H}\to \mathbb{R}
\]
such that for all $x,y\in\mathcal{H}$
\[k(Nx,y)+k(x,Ny)=0\]
and such that, for all $\ell\in\mathbb{N}$, the restriction of $k(.,N^\ell.)$ to $P_\ell\mathcal{H}$ is symmetric positive definite.

\end{defi}

\begin{prop}

A bilinear form $k$ is a polarization of $\mathcal{H}$ if and only if the bilinear form
\[h(.,.):=k(.,w.)
\]
is symmetric positive definite.
\end{prop}

\begin{proof}

\cite[(4.3)]{GuillenHL}

\end{proof}

\subsubsection{Polarized $p$-graded Lefschetz structures}

\begin{defi}

A \emph{$p$-graded Lefschetz structure} consists of a $p$-graded $\mathbb{R}$-vector space $\mathcal{H}$ equipped with nilpotent endomorphisms $N_i$, for $1\leq i\leq p$, commuting pairwise and such that, for all $1\leq i\leq p$ and for all $(\ell_1,...,\ell_{i-1},\ell_{i+1},..,\ell_p)\in \mathbb{Z}^{p-1}$, $(\mathcal{H}_{\ell_1,...,\ell_{i-1},\bullet,\ell_{i+1},..,\ell_p},N_i)$ is a graded Lefschetz structure.
\end{defi}

Thus, a $p$-graded Lefschetz structure admits $p$ actions of $\mathfrak{sl}(2,\mathbb{R})$. 
Identifying the elements of $\mathfrak{sl}(2,\mathbb{R})$ and their images in $\text{End}(\mathcal{H})$, we get, for $1\leq i\leq p$, the endomorphisms $(X_i,Y_i,H_i)\in \text{End}(\mathcal{H})$.

\begin{lemm}

For $i\neq j$, $Y_i,X_i$ and $H_i$ commute with $X_j,Y_j$ and $H_j$. Thus $w_i$ and $w_j$ commute.

\end{lemm}

\begin{proof}

This is a consequence of lemma \ref{proprepr}.

\end{proof}

We define $\boldsymbol{w}:=w_1...w_p.$

\begin{defi}

A \emph{polarization} of a $p$-graded Lefschetz structure is a $p$-graded bilinear form
\[k:\mathcal{H}\otimes\mathcal{H}\to \mathbb{R}
\]
such that for all $x,y\in\mathcal{H}$ and all $1\leq i\leq p$
\[k(N_ix,y)+k(x,N_iy)=0\]
and such that, for all $\boldsymbol{\ell}\in\mathbb{N}^p$, the restriction of  $k(.,N_1^{\ell_1} ... N_p^{\ell_p}.)$ to $P_{\boldsymbol{\ell}}\mathcal{H}$ is symmetric positive definite. Here,
\[P_{\boldsymbol{\ell}}\mathcal{H}:=\mathcal{H}_{\boldsymbol{\ell}}\cap \ker N_1^{\ell_1+1}\cap ...\cap N_p^{\ell_p+1}.
\]

\end{defi}

\begin{prop}\label{caracmultpolar}

A bilinear form $k$ is a polarization of $(\mathcal{H},N_1,...,N_p)$ if and only if, for all $x,y\in\mathcal{H}$ and all $1\leq i\leq p$
\[k(N_i x,y)+k(x,N_i y)=0\]
and the bilinear form
\[h(.,.):=k(.,\boldsymbol{w}.)
\]
is symmetric positive definite.
\end{prop}

\begin{proof}

\cite[(4.3)]{GuillenHL}

\end{proof}

Up to changing the indexes, let's consider that $i=1$ and $j=2$, the endomorphisms ${X}:=X_1+X_2$, ${Y}:=Y_1+Y_2$ and ${H}:=H_1+H_2$ provide an action of $\mathfrak{sl}(2,\mathbb{R})$ such that ${w}=w_1w_2$. Thus, with
\[\mathcal{H}_{\ell,\ell_3,...,\ell_p}:=\bigoplus_{\ell_1+\ell_2=\ell}\mathcal{H}_{\ell_1,\ell_2,\ell_3,...,\ell_p},\]
we get a $(p-1)$-graded Lefschetz structure.

\begin{prop}\label{SLP}

If $(\mathcal{H},N_1,...N_p,k)$ is a polarized $p$-graded Lefschetz structure, then $(\mathcal{H},N=N_1+N_2,N_3...N_p,k)$ is a polarized $(p-1)$-graded Lefschetz structure.

\end{prop}

\begin{proof}

This is an immediate consequence of proposition \ref{caracmultpolar} if we note that $\boldsymbol{w}=w_1w_2w_3...w_p=ww_3...w_p$.

\end{proof}

\subsubsection{Polarized $p$-graded Hodge-Lefschetz structures}

\begin{defi}\label{HLcentrée}

A $p$-graded Hodge-Lefschetz structure centred at $m$ is a $p$-graded Lefschetz structure $(\mathcal{H},N_1,...,N_p)$ such that, for all $\boldsymbol{\ell}\in\mathbb{Z}$, $\mathcal{H}_{\boldsymbol{\ell}}$ admits a real Hodge structure of weight $m+\ell_1+...+\ell_p$ and such that, for all $1\leq i\leq p$, $N_i:\mathcal{H}\to\mathcal{H}(-1)$ is a morphism of Hodge structure. Here $(-1)$ is the Tate twist.

\end{defi}

\begin{defi}

A bilinear form
\[k:\mathcal{H}\otimes\mathcal{H}\to \mathbb{R}
\]

is a polarisation of the $p$-graded Hodge-Lefschetz structures centred at $m$ $(\mathcal{H},N_1,...,N_p)$, if for all $x,y\in\mathcal{H}$ and all $1\leq i\leq p$
\[k(N_i x,y)+k(x,N_i y)=0\]
and for all $\boldsymbol{\ell}\in\mathbb{N}^p$, 
\[P_{\boldsymbol{\ell}}k: P_{\boldsymbol{\ell}}\mathcal{H}\otimes P_{\boldsymbol{\ell}}\mathcal{H} \to \mathbb{R}
\]
is a polarization of the Hodge structure of weight $m+\ell_1+...+\ell_p$ $P_{\boldsymbol{\ell}}\mathcal{H}$. 
\end{defi}

\begin{prop}\label{SHLP}

If $(\mathcal{H},N_1,...N_p,k)$ is a polarized $p$-graded Hodge-Lefschetz structures centred at $m$, then $(\mathcal{H},N=N_1+N_2,N_3...N_p,k)$ is a polarized $(p-1)$-graded Hodge-Lefschetz structure centred at $m$.

\end{prop}

\begin{proof}

The proof is similar to the proof of proposition \ref{SLP}.

\end{proof}

\subsection{Polarized $p$-graded Hodge-Lefschetz modules}

\begin{defi}\label{defiMHL}

We define the category $\textup{MHL}(X,\mathbb{Q},n,p)$ of \emph{$p$-graded Hodge-Lefschetz modules centred at $n$} whose object are the tuples $(\mathcal{M},F,\mathcal{F},N_1,...,N_p)$ such that
\begin{itemize}

\item if $\text{supp}\mathcal{M}=\{x\}$, then there exists a $p$-graded Hodge-Lefschetz structure centred at $n$ (definition \ref{HLcentrée}) $(\mathcal{H}_\mathbb{C},F,\mathcal{H}_\mathbb{Q},N'_1,...,N'_p)$ such that $(\mathcal{M},F,\mathcal{F},N_1,...,N_p)\simeq i_{x*}(\mathcal{H}_\mathbb{C},F,\mathcal{H}_\mathbb{Q},N'_1,...,N'_p)$,
\item if $\dim \text{supp}\mathcal{M}>0$ and $g$ is a locally defined holomorphic function on $X$, then
\[\textup{gr}^M_\bullet\Psi_g(\mathcal{M},F,\mathcal{F},N_1,...,N_p,N) \in \textup{MHL}(X,\mathbb{Q},n-1,p+1).
\]

\end{itemize}

\end{defi}

\begin{defi}\label{defiMHLP}

Let $(\mathcal{M},F,\mathcal{F},N_1,...,N_p)\in \textup{MHL}(X,\mathbb{Q},n,p)$ and let 
\[S:\mathcal{F}\otimes \mathcal{F}\to a_X^!\underline{\mathbb{Q}}_X(-n)
\]
be a pairing. $S$ is a polarization of $(\mathcal{M},F,\mathcal{F},N_1,...,N_p)$ if, for all $1\leq i\leq p$,
\[S\circ (\text{Id}\otimes N_i)+S\circ (N_i\otimes \text{Id} )=0
\]
and for all $\boldsymbol{\ell}\in\mathbb{Z}^{p}$, 
\[S\circ(\text{Id}\otimes N_1^{\ell_1}...N_p^{\ell_p}):P_{\boldsymbol{\ell}}\mathcal{F}\otimes P_{\boldsymbol{\ell}}\mathcal{F}\to a_X^!\underline{\mathbb{Q}}_X(-n-\ell_1-...-\ell_p)
\]
is a polarization of $P_{\boldsymbol{\ell}}\mathcal{F}\in \textup{MH}(X,\mathbb{Q},n+\ell_1+...+\ell_p).$

\end{defi}

We denote $\textup{MHL}^{(p)}(X,\mathbb{Q},n,p)$ the category of polarizable
$p$-graded Hodge-Lefschetz modules centred at $n$.

\section{Compatible filtrations, regular sequences and Koszul complexes}\label{appA}

In this section, we define the compatibility of a family of filtrations and we give several properties of such filtrations. We also describe a way to check the compatibility of filtrations in terms of the flatness of the associated Rees module.

First we define $n$-hypercomplexes corresponding to the naive $n^{uple}$ of \cite[0.4]{DeligneSGA4}.

\begin{de}
 
Let $\mathcal{C}$ be an abelian category, we define by induction the abelian category of 
\emph{n-hypercomplexes} in the following way:
\begin{itemize}
 \item 1-hypercomplexes are the usual complexes of objects in $\mathcal{C}$.
 \item n-hypercomplexes are complexes of (n-1)-hypercomplexes.
\end{itemize}
We will denote $\boldsymbol{C}^n(\mathcal{C})$ the abelian category of n-hypercomplexes composed with objets of $\mathcal{C}$. For instance, the 2-hypercomplexes are the double complexes. A n-hypercomplex is thus equivalent to the following data, for all ${\boldsymbol{k}}\in \mathbb{Z}^n$, an object 
$X^{\boldsymbol{k}}$ of $\mathcal{C}$ and, for all $1\leq i\leq n$, morphisms $d^{(i){\boldsymbol{k}}}:X^{\boldsymbol{k}}\to X^{{\boldsymbol{k}}+\boldsymbol{1}_i}$ satisfying the following properties:
\[\begin{array}{lllccc}
       d^{(i)}\circ d^{(i)}=0 & \text{for all} & i \\
       d^{(i)}\circ d^{(j)}=d^{(j)}\circ d^{(i)} & \text{for all} & (i,j)
   \end{array}\]
for suitable exponents ${\boldsymbol{k}}$.
\end{de}

\subsection{Compatible filtrations}

The following definitions have been introduced by Morihiko Saito in \cite{HM1}

\begin{defi}

Let $A$ be an object in the abelian category $\mathcal{C}$ and $A_1,...,A_n\subseteq A$ sub-objects of $A$. We say that $A_1,...,A_n$ are \emph{compatible sub-objects} of $A$ 
if there exists a $n$-hypercomplex $X$ satisfying:
\begin{enumerate}
\item $X^{\boldsymbol{k}}=0$ if $\boldsymbol{k}\not\in\{-1,0,1\}^n$.
\item $X^\mathbf{0}=A$.
\item $X^{\mathbf{0}-\mathbf{1}_i}=A_i$ for $1\leq i\leq n$.
\item For all $1\leq i\leq n$ and all $\boldsymbol{k}\in\{-1,0,1\}^n$ such that $k_i=0$, the sequence
\[0\to X^{\boldsymbol{k-1}_i}\to X^{\boldsymbol{k}}\to X^{\boldsymbol{k+1}_i}\to 0\]
is a short exact sequence.
\end{enumerate}

\end{defi}

\begin{rema}\label{remcompatibilité}
\begin{itemize}
\item Using the universal properties of short exact sequences, we note that if the sub-objects $A_1,...,A_n$ are compatible, then the $n$-hypercomplex $X$ is uniquely determined. For instance, if $\boldsymbol{k}\in \{-1,0\}^n$ and if $I=\{i;k_i=-1\}\subset\{1,...,n\}$, then \[X^{\boldsymbol{k}}=\bigcap_{i\in I}A_i.\]
\item If $n=1$, the complex $X$ is the short exact sequence
\[0\to A_1\to A\to A/A_1\to 0.\]
\item If $n=2$, two sub-objects $A_1$ and $A_2$ are always compatible and $X$ is the following double complex:
\[\xymatrix{A_1/(A_1\cap A_2) \ar[r]  & A/A_2 \ar[r] & A/(A_1+A_2) \\
A_1 \ar[r]\ar[u] & A \ar[r]\ar[u] & A/A_1\ar[u]\\
A_1\cap A_2 \ar[r]\ar[u] & A_2 \ar[r]\ar[u] & A_2/(A_1\cap A_2).\ar[u]
}\]
\item If $n\geq 3$, sub-objects $A_1,...,A_n$ are not compatible in general.
\item By definition, if $A_1,...,A_n\subseteq A$ are compatible, then for all $I\subset\{1,...,n\}$, the sub-objects $(A_i)_{i\in I}\subseteq A$ are compatible and the corresponding hypercomplex is the $\#I$-hypercomplex $X_I$ with objects $X^{\boldsymbol{k}}$ such that $k_i=0$ for all $i\in I^c$.
\end{itemize}
\end{rema}

\begin{defi}\label{Fcomp}

Let $F_{\bullet}^1,...,F_\bullet^n$ be increasing filtrations indexed by $\mathbb{Z}$ of an object $A$, these filtrations are said to be \emph{compatible} if, for all $\boldsymbol{\ell}\in\mathbb{Z}^n$, the sub-objects $F_{\ell_1}^1,...,F_{\ell_n}^n$ of $A$ are compatible.

\end{defi}

\begin{rema}

\begin{itemize}

\item Following the above remark, avery sub-family of a family of compatible  filtrations is compatible. 
\item We can show that if $F_{\bullet}^1,...,F_\bullet^n$ are compatible, then for all $\ell\in\mathbb{Z}$, the filtrations induced by $F_\bullet^1,...,F_\bullet^{n-1}$ on $\textup{gr}_\ell^{F_n}$ are compatible.
\item If $F_{\bullet}^1,...,F_\bullet^n$ are compatible, the filtrations induced on $F_{\ell_1}^1\cap...\cap F_{\ell_n}^n$ are compatible.
\end{itemize}

\end{rema}

The following proposition correspond to \cite[corollaire 1.2.13]{HM1}.

\begin{prop}\label{commutmultgrad}

Let $F_{\bullet}^1,...,F_\bullet^n$ be compatible filtrations of an object $A$. If we take successively the graded pieces $\textup{gr}_{\ell_{\sigma(j)}}^{F_{\sigma(j)}}$ 
relative to the filtration $F_{\sigma(j)}$ induced on $\textup{gr}_{\ell_{\sigma(j-1)}}^{F_{\sigma(j-1)}}...\textup{gr}_{\ell_{\sigma(1)}}^{F_{\sigma(1)}}A$ for $1\leq j\leq n$, the result does not depend on the permutation $\sigma$ of $\{1,...,n\}$ and is equal to
\[\frac{F_{\ell_1}^1A\cap...\cap F_{\ell_n}^nA}{\sum_j F_{\ell_1}^1A\cap...\cap F_{\ell_{j}-1}^1A\cap...\cap F_{\ell_n}^nA}.\]

\end{prop}

\subsection{Rees's construction}

\begin{defi}\label{Rees}

Let $(A,F^1_{\bullet} A,...,F^n_\bullet A)$ be a filtered object such that the filtrations $F^i_{\bullet}A$ are indexed by $S+\mathbb{Z}$ where $S\subset [0,1[$ is a finite set containing $0$. Let $\varphi:\mathbb{Z}\to S+\mathbb{Z}$ the unique bijection preserving the order and satisfying $\varphi(0)=0$. We define \emph{Rees's graded $\mathbb{C}[y_1,...,y_p]$-module} $\mathscr{A}$ associated to this object as being the following sub-$\mathbb{C}[x_1,...,x_p]$-module of $A\otimes \mathbb{C}[x_1,x_1^{-1}...,x_p,x_p^{-1}]$:
\[\mathscr{A}:=\bigoplus_{\boldsymbol{k}\in\mathbb{Z}^n}
F_{\varphi({\boldsymbol{k}})}A\cdot x_1^{k_i}...x_p^{k_n}\]
where $F_{\varphi({\boldsymbol{k}})}A:=
F_{\varphi(k_1)}A\cap...\cap F_{\varphi(k_n)}A$.

\end{defi}

\begin{prop}\label{compatReesplat}

A graded $\mathbb{C}[x_1,...,x_n]$-module $\mathscr{A}$ is the Rees's module of an object equipped with $n$ compatible filtrations if and only if it is a flat $\mathbb{C}[x_1,...,x_n]$-module.

\end{prop}

\begin{lemm}

A $\mathbb{C}[x_1,...,x_n]$-module $\mathscr{A}$ is flat if and only if every permutation of $x_1,...,x_n$ is a $\mathscr{A}$-regular sequence.

\end{lemm}

\subsection{Regular sequences and Koszul complexes}

Let $\mathcal{R}$ be a commutative ring and $\mathscr{A}$ a $\mathcal{R}$-module. Let $x_1,...,x_p\in \mathcal{R}$, we denote $K(x_1,...,x_p;\mathscr{A})$ the Koszul complex of $\mathscr{A}$ associated to $x_1,...,x_p$ and located between the degrees $-p$ and $0$.

\begin{prop}

The sequence $x_1,...,x_p$ of $\mathcal{R}$ is $\mathscr{A}$-regular if and only if for all $1\leq k\leq p$, the Koszul complex $K(x_1,x_2,...,x_k;\mathscr{A})$ is a left resolution of $\mathscr{A}/(x_1,x_2,...,x_k)\mathscr{A}$.

\end{prop}

\begin{proof}

Let's suppose that $x_1,...,x_p$ is $\mathscr{A}$-regular, then for all $1\leq k\leq p$, the sequence $x_1,...,x_k$ is $\mathscr{A}$-regular. It is a well-known result that the Koszul complex $K(x_1,x_2,...,x_k;\mathscr{A})$ is a left resolution of $\mathscr{A}/(x_1,x_2,...,x_k)\mathscr{A}$. Conversely, we note that the fact that $\mathcal{H}^{-1}(K(x_1,x_2,...,x_k;\mathscr{A}))$ is zero for $1\leq k\leq p$ implies that the morphism of multiplication by $x_k$ 
\[\mathscr{A}/(x_1,x_2,...,x_{k-1})\mathscr{A} \xrightarrow{x_k} \mathscr{A}/(x_1,x_2,...,x_{k-1})\mathscr{A}\]
is injective. This fact for all $1\leq k\leq p$ means that the sequence $x_1,...,x_p$ is $\mathscr{A}$-regular.

\end{proof}

\begin{coro}\label{coroKoszul}

The following statements are equivalent:
\begin{enumerate}
\item Every permutation of $x_1,...,x_p$ is $\mathscr{A}$-regular.
\item Every sub-sequence of $x_1,...,x_p$ is $\mathscr{A}$-regular.
\end{enumerate}

\end{coro}

\begin{proof}

The Koszul complex associated to a sequence of elements does not depend on the ordering.
We then deduce from the above proposition that the statements 1. and 2. are both equivalent to the following statement: For every subset $J\subset \{1,...,p\}$, the Koszul complex $K(\{x_j\}_{j\in J};\mathscr{A})$ is a left resolution of $\mathscr{A}/\mathscr{J}\mathscr{A}$ where $\mathscr{J}$ is the ideal generated by the $x_j$ for $j\in J$.

\end{proof}

\bibliographystyle{smfalpha}
\bibliography{biblio}

\end{document}